\documentclass[11pt,a4paper]{article}

\usepackage[T1]{fontenc}
\usepackage[utf8]{inputenc}
\usepackage[english]{babel}
\usepackage{microtype} %per migliorare il riempimento delle righe
\usepackage{amsmath}  %boh
\usepackage{amsfonts} %comprende il comando \mathbb
\usepackage{mathtools} %comprende anche il pacchetto amsmath
\usepackage[colorlinks=true,allcolors=magenta]{hyperref}
\usepackage{cleveref}
\usepackage{autonum}
\usepackage{bm} %comando \bm{} per usare il grassetto nelle formule
\usepackage{graphicx} %serve per le figure
\usepackage{eurosym}
\usepackage{amsthm}
\usepackage{verbatim} %comando \begin{comment}
\usepackage{enumitem} %numera gli item
\usepackage{color}
\usepackage[left=3cm, right=3cm, top=3cm, bottom=4cm]{geometry}
\usepackage{pgfplots}
\usepackage{caption}
\usepackage{subcaption}
\usepackage{graphicx}
\usepackage{authblk}
\usepackage{blindtext}

\newlength{\imagewidth}

\theoremstyle{plain}
\newtheorem{thm}{Theorem}[section]
\newtheorem{cor}[thm]{Corollary}
\newtheorem{lem}[thm]{Lemma}
\newtheorem{prop}[thm]{Proposition}

\theoremstyle{definition}

\theoremstyle{remark}
\newtheorem{oss}{Remark}

\newcommand{\norm}[1]{\left\lVert#1\right\rVert}
\newcommand{\p}{\mathbb{P}}
\newcommand{\pt}{\tilde{\p}}

\newcommand{\pa}{\p^{(\alpha)}}

\newcommand{\E}{\mathbb{E}}
\newcommand{\Eq}{\E^{(q)}}

\newcommand{\Ea}{\E^{(\alpha)}}
\newcommand{\Stq}{\mc S_t^{(q)}}

\newcommand{\Sta}{\mc S_t^{(\alpha)}}
\newcommand{\Ytq}{Y_t^{(q)}}

\newcommand{\Yta}{Y_t^{(\alpha)}}

\newcommand{\Et}{\tilde{\E}}

\newcommand{\A}{\mathcal{A}}
\newcommand{\G}{\mathcal{G}}

\newcommand{\zeroinf}{(0,\infty)}
\newcommand{\zeroinfc}{[0,\infty)}
\newcommand{\mb}[1]{\mathbb{#1}}

\newcommand{\mc}[1]{\mathcal{#1}}

\newcommand{\Hit}[1]{H{(#1)}}
\newcommand{\NormB}{\norm{B(N-1)}_{\infty}}
\newcommand{\Lxz}{L_{x_0, x_0}}

\newcommand{\Lxy}{L_{x,y}}

\newcommand{\bn}{B\left(N-1\right)}

\newcommand{\h}{h}

\newcommand{\hq}{h_q}
\newcommand{\T}{T}

\newcommand{\hab}{h_{a,b}}
\newcommand{\Uab}{U_{a,b}}
\newcommand{\Uabd}{U^*_{a,b}}
\newcommand{\roab}{\rho_{a,b}}
\newcommand{\nuab}{\nu_{a,b}}
\newcommand{\Mab}{\mc M_{a,b}}
\newcommand{\Eg}{\mc E^g}

\newcommand{\Chat}{\mc C_0 \left([0, \infty)\right)}

\newcommand{\Chatab}{\mc C_0 \left([a,b)\right)}
\newcommand{\Chatabp}{\mc C_0^{^+} \left([a,b)\right)}
\newcommand{\At}{\tilde{\mc A}}
\newcommand{\sab}{\sigma(a,b)}
\newcommand{\qbn}{q_{_{BN}}}
\newcommand{\qg}{q_g}

\title{A probabilistic view on the long-time behaviour of growth-fragmentation semigroups with bounded fragmentation rates}
\author{Benedetta Cavalli\footnote{Institut für Mathematik, Universität Zürich, Switzerland. E-mail address: \href{mailto:name.surname@gmail.com}{benedetta.cav@gmail.com}.}}

\begin{document}
	\maketitle
	\vspace{-0.6 cm}
	
	\begin{abstract}
The growth-fragmentation equation models systems of particles that grow and reproduce as time passes. An important question concerns the asymptotic behaviour of its solutions. Bertoin and Watson ($2018$) developed a probabilistic approach relying on the Feynman-Kac formula, that enabled them to answer to this question for sublinear growth rates. This assumption on the growth ensures that microscopic particles remain microscopic. In this work, we go further in the analysis, assuming bounded fragmentations and allowing arbitrarily small particles to reach macroscopic mass in finite time. We establish necessary and sufficient conditions on the coefficients of the equation that ensure Malthusian behaviour with exponential speed of convergence to the asymptotic profile. Furthermore, we provide an explicit expression of the latter.
\end{abstract}

\vspace{0.4 cm}


\textit{Classification MSC:} 34K08 , 35Q92, 47D06, 47G20, 45K05, 60G51, 	60J99  

\normalsize
	
\section{Introduction} 
Imagine a population of individuals that grow and reproduce as time proceeds, in such a way that the evolution of each individual is independent from the others. 

The \textit{growth-fragmentation equation} is the key equation that has been used in the field of structured population dynamics to model such systems. It was first introduced to describe cells dividing by fission \cite{BA67} and, sequently, it has also been used to model neuron networks \cite{KPS13}, polymerization \cite{CLODLMP09,PPS13}, the TCP/IP window size protocol for the internet \cite{BMR02} and many other systems sharing the dynamics described above. The common point is that the ``particles'' under concern (cells, polymers, dusts, etc.) are well-characterized by their mass (or ``size''), \textit{i.e.}, a one-dimensional quantity that grows over time at a certain rate (depending on the mass) and that is distributed among the offspring when a dislocation event occurs. In this work, we do not  assume conservation of mass at dislocation events. This means that some of the mass may be lost or gained during a dislocation.

The main quantity of interest is the concentration of particles of mass $x > 0$ at time $t \geq 0$, denoted by $u_t(x)$. The growth-fragmentation equation describes the evolution of $u_t(x)$ and can be obtained either by a mass balance, in a similar way as for fluid dynamics \cite{BSC11,MD86}, or by considering the Kolmogorov equation for the underlying jump process  \cite{CLOEZ17,DHNR15}:
\begin{equation}
\label{gfe}
\begin{cases}
\partial_t u_t(x) + \partial_x (\tau(x) u_t(x)) + B(x)u_t(x) =  \int_x^{\infty} B(y) k(y, x)u_t(y) dy, \quad \quad x > 0 \\
u_t(0)=0, \\
u_0 (x)  \quad \text{prescribed}.
\end{cases}
\end{equation}

Here, the \textit{growth rate} 
\begin{equation}
\label{tauLipschitz}
\tau : [0,\infty) \to (0,\infty) \quad \text{is a continuously differentiable function},
\end{equation}
the \textit{fragmentation rate}
\begin{equation}
\label{Bcontinuousbounded}
B : [0,\infty) \to [0,\infty) \quad \text{is continuous and bounded},
\end{equation}
and the \textit{fragmentation kernel} $k(x, y): [0, \infty) \times [0, \infty) \to [0, \infty) $ is such that 
\begin{equation}
\label{kcontinuous}
\text{the map} \quad \begin{cases}
(0, \infty) \to L^1(dy) \\
\quad x \quad  \mapsto \; k(x,y)
 \end{cases}
\quad  \text{is contintuous, with} \quad k(x,y) = 0 \quad \forall y \geq x.
\end{equation}
Moreover, we define the function
\begin{equation}
\label{equationnumberofchildren}
N(x) \coloneqq \int_0^x k(x, y) dy,
\end{equation}
and we assume that
\begin{equation}
\label{Nconitnuousbounded}
N : [0,\infty) \to [1,\infty) \quad \text{is continuous and bounded}.
\end{equation}
In words, particles of size $x>0$ grow with speed $\tau(x)$ and divide with division rate $B(x)$. When a particle of size $x$ splits, it produces an average of $N(x)$ smaller particles and $B(x)k(x,y)$ is the rate of birth of a particle having size $y$ from a particle with size $x$. 

In this work, we rather deal with the weak form of the growth-fragmentation equation \eqref{gfe}, that is
\begin{equation}
\label{wgfe}
\frac{d}{dt} \langle \mu_t, f \rangle = \langle \mu_t, \mc A f \rangle.
\end{equation}
Here, $\mu_t(dx) \coloneqq u_t(x) dx$, the function $f$ is smooth with compact support and $\langle \mu , g \rangle$ denotes $\int g(x) \mu(dx)$ for any measure $\mu$ and any function $g$, whenever it makes sense.
The operator $\mc A$, called \textit{growth-fragmentation operator}, has the form
\begin{equation}
\label{definitiongfoperator}
\A f(x) = \tau(x)f'(x) + B(x) \int_0^x f(y) k(x,y) dy - B(x) f(x),
\end{equation}
and it is defined on some domain $\mc D_{\mc A}$ of smooth functions, which will be made explicit in Section \ref{Section3}. 
Proper assumptions on the coefficients $\tau$, $B$ and $k$, specified in Section \ref{Section3}, guarantee that $\mc A$ is the infinitesimal generator of a unique strongly continuous positive semigroup $(T_t)_{t \geq 0}$. In this case, \eqref{wgfe} has a unique solution, given by
\begin{equation}
\label{weaksolutiongfe}
\langle u_t, f \rangle = \langle u_0, T_tf \rangle.
\end{equation}
\begin{comment}
In this setting, we introduce the measure $\mu_t (x, dy)$ on $\zeroinf$ such that
\begin{equation}
T_t f(x) = \int_{\zeroinf} f(y) \mu_t (x, dy) = \langle \mu_t (x, \cdot), f \rangle 
\end{equation}
describes the concentration at time $t$ of individuals of mass $y$ when one starts at time $0$ from a unit concentration of individuals of mass $x$, i.e. $\mu_0(x, dy) = \delta_x (dy)$. \\
\end{comment}

\begin{comment}
Under proper assumptions on the rates $\tau$, $B$ and $k$ which will be specified in due time, $\mc A$ is the infinitesimal generator of a unique strongly continuous positive semigroup $(T_t)_{t \geq 0}$, that is 
\begin{equation}
    \frac{dT_tf(x)}{dt}= \mc A (T_t f)(x), \quad \quad \quad f \in \mc D_A.
\end{equation}
\end{comment}
Note that the weak form \eqref{wgfe} enables to extend the analysis to cases where the concentration of particles is not absolutely continuous w.r.t.\ the Lebesgue measure. In particular, we are able to treat initial conditions of Dirac type. In this setting, for all $x > 0$, the measure\footnote{Note that it exists and is unique thanks to the Riesz-Markov representation theorem.} $\mu_t (x, dy)$ on $\zeroinf$ such that
\begin{equation}
T_t f(x) = \int_{\zeroinf} f(y) \mu_t (x, dy) = \langle \mu_t (x, \cdot), f \rangle,
\end{equation}
describes the concentration at time $t$ of individuals of mass $y$ when one starts at time $0$ from a unit concentration of individuals of mass $x$, \textit{i.e.}\ $\mu_0(x, dy) = \delta_x (dy)$. 

In general, one cannot expect to have an explicit expression for the growth-fragmentation semigroup $(T_t)_{t \geq 0}$ and, motivated by several applications in mathematical modelling, many works are concerned with its behaviour for large times. Typically\footnote{This evidence has been supported by many empirical results, see for example \cite{SMPT05}.}, one expects that, under proper assumptions on the growth and fragmentation rates, there exist $\rho \in \mathbb{R}$, a Radon measure $\nu(dx)$, usually called \textit{asymptotic profile}, and a positive function $\h$ such that
\begin{equation}
\label{convergenceofthesemigroup}
\lim_{t \to \infty} e^{-\rho t} T_t f(x) = h(x) \langle \nu, f \rangle, \quad x > 0,
\end{equation}
at least for every continuous and compactly supported function $f: (0, \infty) \to \mathbb{R}$.
In the literature, the above convergence is often referred to as \textit{Malthusian behaviour}. When it holds, a further important question concerns the speed of convergence. To understand why, consider for example the case in which \eqref{convergenceofthesemigroup} holds with $\rho > 0$. This would imply that the concentration of particles grows exponentially in $t$, albeit, in reality, due to several effects such as the scarcity of space and resources, an indefinite exponential growth is not possible. As a consequence, the growth fragmentation equation is reliable only for rather early stages of the evolution of the population, and the exponent $\rho$ and the asymptotic profile are meaningful only when $e^{-\rho t} T_t $ converges to the asymptotic profile fast enough. Thus, one wishes to establish the so-called \textit{exponential convergence}, \textit{i.e.},
\begin{equation}
\label{definitionexponentialconvergence}
e^{-\rho t} T_t f(x) = h(x) \langle \nu, f \rangle + o (e^{-\beta t}),  \quad x>0,
\end{equation}
for some $\beta > 0$.
\begin{comment}
The tool that has been mostly used in the literature to investigate \eqref{convergenceofthesemigroup} is the spectral theory of semigroups and operators. The cornerstone of this approach consists in proving the existence of positive eigenelements associated to the leading eigenvalue, say $\rho \in \mathbb{R}$, of the operator $\mc A$ and its dual $\mc A^*$, \textit{i.e.}, a Radon measure $\nu$ and a positive function $h$ such that 
\begin{equation}
\label{eigenelements}
\mc A h = \rho h, \quad \mc A^* \nu = \rho \nu, \quad \text{and} \quad  \langle \nu, h \rangle = 1.
\end{equation} 
The triplet $(\rho, \h, \nu)$ is said to be a solution to the \textit{eigenvalue problem} for $\mc A$. Proper assumptions on the growth and fragmentation rates that ensure existence and uniqueness of a solution to the eigenvalue problem have been established by several authors, for example Mischler and Scher \cite{MS16}, Doumic and Gabriel \cite{DG10} and Mischler \cite{M06}.
\end{comment}

The tool that has been mostly used in the literature to investigate \eqref{convergenceofthesemigroup} is the spectral theory of semigroups and operators. The cornerstone of this approach consists in proving the existence of a solution to the so-called \textit{eigenvalue problem} for $\mc A$, namely a triplet $(\rho, \h, \nu)$ that satisfies
\begin{equation}
\label{eigenelements}
\mc A h = \rho h, \quad \mc A^* \nu = \rho \nu, \quad \text{and} \quad  \langle \nu, h \rangle = 1,
\end{equation} 
with $\mc A^*$ being the dual operator of $\mc A$, $\rho$ the leading eigenvalue of $\mc A$ and $\mc A^*$, $\nu$ a Radon measure and $\h$ a positive function. Proper assumptions on the growth and fragmentation rates that ensure existence and uniqueness of a solution to the eigenvalue problem have been established by several authors, for example Mischler and Scher \cite{MS16}, Doumic and Gabriel \cite{DG10} and Michel \cite{M06}.

Once \eqref{eigenelements} is proved, several techniques can be used to derive \eqref{convergenceofthesemigroup}. For instance, Cáceres at al. \cite{CCM11} used dissipation of entropy and entropy inequalities methods to prove convergence to an asymptotic profile for constant and linear growth rate, while Perthame \cite{PERTHAME07} and Michel et al. \cite{MMP05} relied on the general relative entropy method. Mischler and Scher \cite{MS16} developed a splitting technique that allows to formulate a Krein-Rutman theorem. They also provided a punctual survey on the spectral analysis of semigroups. Finally, exponential rate of convergence is essentially equivalent to the existence of a spectral gap \cite{PR05, LP09, CCM11, MS16}.

Doumic and Escobedo \cite{DE16} and Bertoin and Watson \cite{BW16} used the Mellin transform to analyse the so-called \textit{critical case}, where the strategy outlined above cannot be applied as there is no solution to the eigenvalue problem \eqref{eigenelements}. Indeed, in this case, even though it is possible to find positive eigenelements for the growth-fragmentation operator, the integrability condition $\langle \nu, h \rangle = 1$ is not satisfied and \eqref{eigenelements} fails. 

In more recent years, the growth-fragmentation equation \eqref{gfe} have been studied with probabilistic methods. For instance, some authors, including Bardet et al.\ \cite{BCGMZ13}, Bouguet \cite{BOUGUET18} and Chafaï et al.\ \cite{CMP10}, relied on probabilistic techniques to study the conservative version of \eqref{gfe}, in which the total mass of the system is conserved. 

Bertoin and Watson \cite{BERT18,BW18} developed a probabilistic approach to \eqref{convergenceofthesemigroup}, relying on a Feynman-Kac representation of the growth-fragmentation semigroup, that circumvents the spectral theory of semigroups. They could establish necessary \cite{BW18} and sufficient \cite{BERT18} conditions for the Malthusian behaviour with exponential speed of convergence when the growth rate is continuous and sublinear, \textit{i.e.}, $\sup_{x >0} \tau(x)/x < \infty$. With a similar approach, Cavalli \cite{BC19} obtained necessary and sufficient assumptions for exponential convergence in the case of homogeneous fragmentations (the rate at which particles split not depend on the size) and piecewise-linear growth rate. One of the main benefits of this approach is that it also provides a probabilistic representation of the quantities of interest (asymptotic profile, exponent $\rho$, etc.). 

A common point in the cases studied by Bertoin, Watson and Cavalli, is that microscopic particles remain microscopic. More precisely, the time after which a particle of infinitesimal mass growing at speed $\tau$ reaches a fixed mass (say $1$ for the sake of simplicity), namely

\begin{equation}
\label{definitionofT}
\T \coloneqq \int_0^{1} \frac{dx}{\tau(x)},
\end{equation} 

is infinite. In this work, on the contrary, we focus on
\begin{equation}
    \label{Tfinite}
    T < \infty,
\end{equation}
\textit{i.e.}, particles with arbitrarily small masses may become macroscopic after a bounded time. 
We further assume that particles with finite mass cannot reach infinite mass in finite time, \textit{i.e.},
\begin{equation}
\label{Timetoinfinity}
\int_1^{\infty} \frac{dx}{\tau(x)} = \infty.
\end{equation}
We stress that, unlike the case in \cite{BERT18,BW18, BC19}, in our model it is crucial to assume bounded fragmentations.

Just as in \cite{BERT18, BW18}, our analysis relies on a Feynman-Kac representation of the semigroup $(T_t)_{t \geq 0}$ in terms of an instrumental Markov process $X=(X_t)_{t \geq 0}$. Its infinitesimal generator is
\begin{align}
\label{definitionofg}
\G f (x) & =  \tau(x)f'(x) + B(x) \int_0^x \left(f(y)- f(x) \right)  k(x,y) dy,
\end{align}
and it is closely related to the growth-fragmentation operator $\mc A$. However, the Markov process we rely on is different from the one used in \cite{BERT18, BW18}, letting us treat different situations. In their case, in fact, the dynamics of $X$ can be seen as the dynamics of the mass of a distinguished individual in the population, such that, at every dislocation event, the distinguished daughter is chosen among the siblings by size-biased sampling. In particular, their process jumps at the same rate as the one at which the individuals of the population reproduce. In our case, the process $X$ jumps at rate $BN$, while the particles in the system reproduce at rate $B$. Thus, $X$ cannot be seen as a ``well-chosen'' particle in the system. 

From \eqref{definitionofg}, we see that the trajectory $t \mapsto X_t$ is driven by the deterministic flow velocity $\tau$ between consecutive jumps and that the jumps are the only source of randomness\footnote{In this case we say that $X$ is piecewise deterministic, see \cite{D84} for a complete introduction.}. Assumptions \eqref{Bcontinuousbounded} and \eqref{Nconitnuousbounded} guarantee that the total jump rate of $X$ is bounded, so the jumps never accumulate. In the rest of the work, we assume that, for every $x>0$, there exists $\alpha < x < \beta$ with
\begin{equation}
\label{conditionequivalenttoirreducibility}
 \int_{\alpha}^x k(\beta, y) dy >0,
\end{equation}
which is equivalent to the irreducibility of the process $X$ in $\zeroinf$, as it is shown in Section \ref{Section3}. Comparing \eqref{definitiongfoperator} and \eqref{definitionofg}, we get the Feynman-Kac representation\footnote{We refer to Section \ref{Section3} for a rigorous proof.}
\begin{equation}
\label{FKrep}
T_t f(x) = \E_x \left( \mc{E}_t f(X_t) \right), \quad \quad t \geq 0, \quad x\geq0,
\end{equation}
with
\begin{equation}
\label{definitionofE}
\mc{E}_t \coloneqq \exp \left( \int_0^t  B(X_s) \left(N(X_s) - 1 \right) ds \right), \quad \quad t \geq 0,
\end{equation}
where $\p_x$ (resp.\ $\E_x$) is the probability measure (resp.\ the expectation) when the process $X$ is conditioned to start at $X_0=x$.

Even though \eqref{FKrep} is not quite explicit in general, it is of great help to study the behaviour of $T_t$ as $t \to \infty$. A fundamental role is played by the function
\begin{equation}
\label{Lxy}
L_{x,y} (q) \coloneqq \E_x \left(e^{-qH(y)} \mc E_{H(y)}, \; H(y) < \infty  \right), \quad \quad q \in \mb R, \quad x,y \in \zeroinf,
\end{equation}
where $H(y)$ denotes the first hitting time of $y$ by $X$. An important property of $L$ (we refer to Section \ref{Section2} for an extensive analysis) is that it is non-increasing and convex. This allows to fix $x_0>0$ and define the \textit{Malthus exponent} 
\begin{equation}
\label{definitionoflambda}
\lambda \coloneqq \inf \{ q \in \mathbb{R} \; : \; \Lxz (q) <1 \}.
\end{equation}

\subsection*{Main results}
The main contribution of the present work is to provide sufficient conditions in terms of the coefficients $\tau$, $B$ and $k$ that ensure exponentially fast convergence of $e^{-\lambda t}T_t$ to an asymptotic profile. Moreover, we also give an explicit expression of the latter.

\begin{thm}
\label{Theorem2}
Assume \eqref{tauLipschitz}, \eqref{Bcontinuousbounded}, \eqref{kcontinuous}, \eqref{Nconitnuousbounded}, \eqref{Timetoinfinity}, \eqref{conditionequivalenttoirreducibility} and the forthcoming \eqref{vagueconvergenceofthekernel}. If the Malthus exponent $\lambda$ and the rates $B$ and $N$ satisfy 
\begin{equation}
\label{conditionlimsupatinfinity}
\limsup_{x \to \infty} B(x)(N(x)-1) < \lambda,
\end{equation}
then the Malthusian behaviour with exponential convergence \eqref{definitionexponentialconvergence} holds, with $\rho= \lambda$,
\begin{equation}
\label{definitionofh}
\h(x) \coloneqq L_{x, x_0}(\lambda), \quad x > 0,
\end{equation}
and
\begin{equation}
\label{definitionofnu}
\nu(dx)\coloneqq \frac{dx}{\h(x) \tau(x) |L'_{x,x}(\lambda)|}, \quad x > 0.
\end{equation}
\end{thm}

It is further interesting to discuss the criterion  \eqref{conditionlimsupatinfinity}. On one hand, it may seem a bit surprising, as it is often assumed in the literature that fragmentations of big particles should be strong enough to counterbalance the growth. However, an heuristic interpretation can be given by making a comparison with branching processes. Consider a system in which particles die with rate $B$ and, when a particle of size $x$ dies, it is replaced by an average of $N(x)$ particles. The quantity $B(x) \left( N(x) -1 \right)$ can be seen as the average ``increase'' in the number of particles of the system that arises from the death of a particle of size $x$, whilst the Malthus exponent $\lambda$ represents the long time increase in the number of particles of the system. Condition \eqref{conditionlimsupatinfinity} says that, when particles are large enough, they mostly produce a number of particles that is smaller than the average. So, roughly speaking, the main contribution to the evolution of $\mu_t$ comes from particles that stay in some compact subset of $[0, \infty)$. Condition \eqref{conditionlimsupatinfinity} then does not come as a surprise, as it is well known that compactness plays a key role in establishing Malthusian behaviour, for example in the Krein-Rutman setting.

Condition \eqref{conditionlimsupatinfinity} may still seem unsatisfactory, since it depends not only on the coefficients, but also on the Malthus exponent $\lambda$. However, in many cases, it can be made much more explicit. In particular, if $X$ is recurrent and $B,N$ are not constant, then $ \inf_{x>0} B(x) \left( N(x)- 1 \right) < \lambda$ (see Proposition \ref{Propositionsimpleboundsonlambda}). Thus, when $X$ is recurrent, \eqref{conditionlimsupatinfinity} is surely fulfilled if
\begin{equation}
\label{conditionwhenrecurrent}
\lim_{x \to \infty} B(x)(N(x)-1)  = \inf_{x>0} B(x)(N(x)-1).
\end{equation}
This enables to find explicit conditions for the Malthusian behaviour in the important case when the fragmentation kernel is self-similar, \textit{i.e.},
\begin{equation}
\label{selfsimilarfragmentationkernel}
k(x,y)= \frac{1}{x} k_0 \left(\frac{y}{x}\right), \quad 0<y<x,
\end{equation}
where $k_0 \in L^1([0,1])$. For all $r \in \mathbb{R}$, we define
\begin{equation}
M_r \coloneqq \int_0^1 z^r k_0 (z) dz.
\end{equation}
It is easy to check that, in this case, $N(x) = M_{_0}$ for all $x>0$. 

\begin{thm}
\label{Theorem3}
Assume \eqref{tauLipschitz}, \eqref{Bcontinuousbounded}, \eqref{kcontinuous} and \eqref{Timetoinfinity}. Assume further that \eqref{conditionwhenrecurrent} holds, that
\begin{equation}
\label{assmoments} 
\text{there exist} \; a,b>0 \; \text{such that} \;  M_a < M_{_0} \; \text{and} \; M_{-b} < \infty,
\end{equation}
and that there exists $x_\infty > 0$ such that, for all $x \geq x_{\infty}$,
\begin{equation}
\label{conditionrecurrencessinf}
\frac{\tau(x)}{xB(x)} \leq \frac{1}{a} \left(M_{_0} - M_a \right).
\end{equation} Then, the Malthusian behaviour with exponential speed of convergence \eqref{definitionexponentialconvergence} holds with $\rho = \lambda$ and $\h$ and $\nu$ as in Theorem \ref{Theorem2}.
\end{thm}
We stress that condition \eqref{conditionrecurrencessinf} is quite natural and it can be interpreted as a balance between the growth and the fragmentation of large particles.

\subsection*{Related results}
The growth-fragmentation equation with self-similar fragmentation rate has been extensively studied in the literature and it is interesting to compare our results with the previous ones. 

To start, Bouguet \cite{BOUGUET18} investigated positive recurrence for the family of piecewise-deterministic Markov processes that arise in our analysis. Sufficient conditions for positive recurrence are provided in the case in which $\tau$ and $B$ behave as power functions of the size in a neighbourhood of $0$ and $\infty$. In addition to \eqref{conditionrecurrencessinf}, in \cite{BOUGUET18} a balance between growth and fragmentation of very small particles is also assumed (assumption $(2.5)$ in \cite{BOUGUET18}). In our case, this extra condition is instead a direct byproduct of our setting (see the proof of Theorem \ref{Theorem3} for further details).

Doumic Jauffret and Gabriel \cite{DG10} obtained conditions to ensure the existence of eigenelements and Malthusian behaviour (by general entropy method). Their assumption
\begin{equation}
\lim_{x \to \infty} \frac{x B(x)}{\tau(x)} = + \infty,
\end{equation}
clearly implies our condition \eqref{conditionrecurrencessinf}. Furthermore, while they also assume \eqref{Tfinite} for $B(0) >0$, we recall that we don't assume conservation of mass, that is condition $(6)$ in \cite{DG10}.

Bernard and Gabriel \cite{BG17} provided sufficient conditions for the existence of a solution to the eigenvalue problem \eqref{eigenelements}  in the self-similar case, assuming bounded fragmentations. Our condition on the behaviour of $\tau$ at $\infty$ is less restrictive than theirs, as they assume that there exist $ \bar{\alpha} \leq \alpha < 1$ such that
\begin{equation}
c_1 x^{\bar{\alpha}} < \tau(x) < c_2 x^{\alpha}.
\end{equation}
Similarly, for the fragmentation rate, instead of their assumption of 
$B(x)$ constant for large $x$, we require continuity conditions for $B$. 

We also mention \cite{BCGM19}, in which the authors analyse the self-similar case assuming constant growth rate.

Finally, our results should be considered together with the ones obtained by Bertoin \cite{BERT18}, who also analyses the self-similar case (see paragraph $3.5$), but in a complementary framework. In fact, in \cite{BERT18}, \eqref{Tfinite} does not hold and the fragmentations may be unbounded. Condition $(32)$ in \cite{BERT18} is the the same as our condition \eqref{conditionrecurrencessinf}. However, they again  require a balance between growth and fragmentation of very small particles, that is always verified under our assumptions. 

We conclude mentioning that a possible approach to the study of the asymptotic behaviour of the growth-fragmentation equation may be developed with the help of quasi-stationary distributions. We refer to \cite{CV16} and \cite{CV17} for a comprehensive introduction on the topic.

\subsubsection*{Outline of the paper} The article is organised as follows. In Section \ref{Section2} we present some general results on Markov processes with only negative jumps, that will be used in the rest of the work. In Section \ref{Section3} we establish existence and uniqueness of the growth-fragmentation semigroup, as well as its  Feynman-Kac representation. Moreover, we provide a characterization of the Malthusian behaviour in Theorem \ref{Theorem1}. Section \ref{Section4} is devoted to the proofs of Theorem \ref{Theorem2} and Theorem \ref{Theorem3}. Finally, we provide some examples in Section \ref{Section5}.

\section{Background on the instrumental Markov process}
\label{Section2}

In this section we aim to present in a more general setting some of the ideas and techniques used by Bertoin and Watson \cite{BW18}, Bertoin \cite{BERT18} and Cavalli \cite{BC19} to study properties of Markov processes of the type \eqref{definitionofg}. The results obtained will be of great use in the next sections.
For the sake of simplicity, we use here the same notation that was used in the introduction (for instance the notation $X$, $\p_x$, $\E_x$ or $H$), even though we are considering slightly more general processes.

\subsection{Setup}
This section concerns processes with infinitesimal generator of the type \eqref{definitionofg}. However, rather than working with the analytic expression of the generator, it will be more useful for our analysis to focus on the path properties of this kind of processes. 

We consider a Markov process $X$ on $[0, \infty)$ that satisfies the following. The trajectory $t \mapsto X_t$ follows a strictly increasing deterministic flow between consecutive jumps and the jumps are the only source of randomness. The total jump rate remains bounded, so the jumps never accumulate. Denote by $\p_x$ the law of $X$ started at $x \geq 0$, by $\E_x$ the corresponding expectation and let
\begin{equation}
\label{definitionhittingtimes}
\Hit y \coloneqq \inf \{t>0\; : \; X_t =y\},
\end{equation}
be the first hitting time of $y > 0$. We make the following assumptions:

\begin{enumerate}[label=(A\arabic*)]
%\item $0$ is an \textit{entrance boundary} for $X$, \textit{i.e.}, $\p_0(H(0) < \infty) = 0$, or $0$ is an \textit{absorbing state}, \textit{i.e.}, $\p_0(X_t = 0 \; \text{for all} \; t>0)=1$ or \textcolor{magenta}{$\p_0(H(y) < \infty) = 0$ for all $y>0$};
\item \label{Sec2upwardskip} $X$ has no positive jumps (\text{upward skipfree});
\item \label{Sec2irreducible} $X$ is irreducible in $(0, \infty)$; \textit{i.e.}, $\p_x(H(y) < \infty) > 0$ for all $x, y > 0$;
\item \label{Sec2entranceboundary} $0$ is an \textit{entrance boundary}, \textit{i.e.}, $\p_0(H(x) < \infty) > 0$ and $\p_x(H(0) < \infty) = 0$, for all $x > 0$.
\end{enumerate}
We notice that, since the deterministic flow is strictly increasing and the jumps don't accumulate, we have that
\begin{enumerate}[label=(A\arabic*)]
\setcounter{enumi}{3}
\item \label{Sec2positiverettimes}  return times are almost surely strictly positive in $(0,\infty)$, \text{i.e.}, $\p_x(H(x)>0)=1$ for all $x > 0$.
\end{enumerate}

\begin{oss}
We stress that the properties \ref{Sec2upwardskip}-\ref{Sec2positiverettimes} hold when $X$ has generator given by \eqref{definitionofg} under the assumptions outlined in the Introduction, as we will show in Section \ref{Section3}.
\end{oss}

Let $g: (0, \infty) \to (0, \infty)$ be a measurable and bounded function and define the random functional 
\begin{equation}
\Eg_t = \exp \left(  \int_0^t g(X_s) ds  \right), \quad \quad t \geq 0.
\end{equation}

We aim to construct some martingales and a family of supermartingales connected to the process $X$ and the functional $\Eg_t$. 

\subsection{A Laplace transform}
We start by defining the Laplace transform 
\begin{equation}
L_{x,y}(q) \coloneqq \E_x \left( e^{-q \Hit y} \Eg_{\Hit y} , \Hit y < \infty \right), \quad \quad q \in \mathbb{R}, \; x \geq 0, y >0.
\end{equation}

First of all, \ref{Sec2irreducible} and \ref{Sec2entranceboundary} imply that $\p_x( \Hit y < \infty) > 0$ for all $x \geq 0$ and $y>0$. Furthermore, on the event $\{ \Hit y < \infty \}$, the functional $\Eg_{\Hit y}$ is strictly positive, and so $L_{x,y}(q) \in (0, \infty] $. Straightforward arguments show that the function $\Lxy: \; \mathbb{R} \to (0, \infty] $ is non increasing, convex, and right-continuous at the boundary points of its domain (monotone convergence). Moreover, for every $q > \norm{g}_{\infty}$, we have $e^{-qt}\Eg_t \leq 1$ and, a fortiori, $L_{x,y}(q) <1$. More precisely,  
\begin{equation}
\lim_{q \to - \infty} L_{x,y}(q)  = \infty \quad \text{and} \quad \lim_{q \to + \infty} L_{x,y}(q)  = 0. 
\end{equation}

Thanks to this property, we can fix $x_0 >0$ arbitrarily and define
\begin{equation}
\lambda \coloneqq \inf \{ q \in \mathbb{R} \; : \; \Lxz (q) <1 \}.
\end{equation}
The definition of $\lambda$ does not depend on the choice of $x_0$ (see Proposition $3.1$ in \cite{BW18}). We state some elementary bounds for $\lambda$ in terms of the function $g$. The proof is similar to that of Proposition $3.4$ in \cite{BW18} and details are left to the reader.

\begin{prop}
\label{Propositionsimpleboundsonlambda}
Assume \ref{Sec2upwardskip}-\ref{Sec2positiverettimes}. Then the following hold.
\begin{enumerate}[label=(\roman*)]
\item It always holds that $\lambda \leq \norm{g}_{\infty}$.
\item If $X$ is recurrent, then $\lambda \geq \inf_{x >0} g(x)$ and, in particular, $\lambda \geq 0$. The strict inequality holds except possibly when $g$ is constant. 
\item If $X$ is positive recurrent on $\zeroinf$ with stationary measure $\pi$, then 
\begin{equation}
\lambda \geq \langle \pi, g \rangle.
\end{equation}
\end{enumerate}
\end{prop}

Due to right continuity, it always holds that $L_{x_0,x_0}(\lambda) \leq 1$. We consider now the following assumption 
\begin{equation}
\label{Lequal1}
    L_{x_0,x_0}(\lambda) = 1,
\end{equation}
and the stronger one
\begin{equation}
\label{Lbiggerthan1}
\text{there exists} \; q \in \mathbb{R} \; \text{with} \; L_{x_0, x_0}(q) \in (1, \infty).
\end{equation}

Notice that condition \eqref{Lbiggerthan1} implies not only \eqref{Lequal1}, but also that
\begin{equation}
\label{Lfinitederivative}
    L'_{x_0, x_0}(\lambda) > - \infty.
\end{equation}

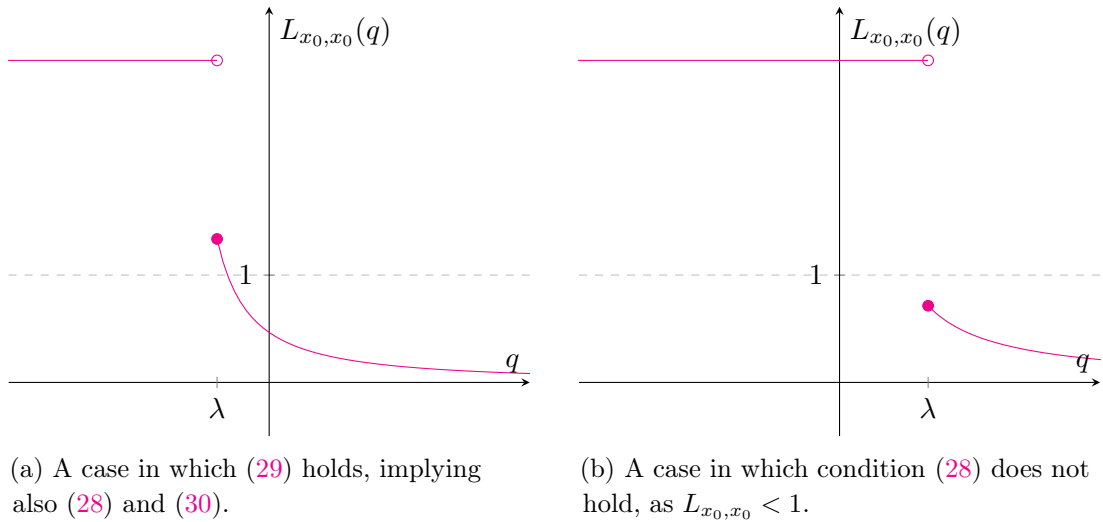
\begin{figure}[h]
\label{Fig1}
\centering
\imagewidth=0.30\textwidth
\captionsetup[subfigure]{width=1.5\imagewidth,justification=raggedright}
\begin{subfigure}{.5\textwidth}
  \centering
  \begin{tikzpicture}
\begin{axis}[
	axis lines = center,
	xlabel={$q$},
	ylabel={$L_{x_0, x_0}(q)$},
    xmin=-5, xmax=5,
    ymin=-0.5, ymax=3.5,
    xtick={0},
    ytick={1},
    extra x ticks={-1},
	extra x tick labels={$\lambda$},
    %legend pos=north west,
    ymajorgrids=true,
    grid style=dashed,
]
%Below the red parabola is defined
\addplot [
    domain=-1:10, 
    samples=100, 
    color=magenta,
]
{1/(x+1.8)^1.3};
%Here the blue parabloa is defined
\addplot [
    domain=-10:-1, 
    samples=100, 
    color=magenta,
    ]
    {3};
%\addlegendentry{$x^2 + 2x + 1$}
\addplot[mark=*, magenta] coordinates {(-1, 1.3365)};
\addplot[mark=o, magenta] coordinates {(-1, 3)};
\end{axis}
\end{tikzpicture}
 \captionof{figure}{A case in which \eqref{Lbiggerthan1} holds, implying also \eqref{Lequal1} and \eqref{Lfinitederivative}.}
 % \label{fig:test1}
\end{subfigure}%
\hfill
% or \hspace{0.3\textwidth}
\begin{subfigure}{.5\textwidth}
  \centering
  \begin{tikzpicture}
\begin{axis}[
	axis lines = center,
	xlabel={$q$},
	ylabel={$L_{x_0, x_0}(q)$},
    xmin=-5, xmax=5,
    ymin=-0.5, ymax=3.5,
    xtick={0},
    ytick={1},
    extra x ticks={1.7},
	extra x tick labels={$\lambda$},
    %legend pos=north west,
    ymajorgrids=true,
    grid style=dashed,
]
%Below the red parabola is defined
\addplot [
    domain=1.7:10, 
    samples=100, 
    color=magenta,
]
{1/(x-0.3)};
%Here the blue parabloa is defined
\addplot [
    domain=-10:1.7, 
    samples=100, 
    color=magenta,
    ]
    {3};
%\addlegendentry{$x^2 + 2x + 1$}
\addplot[mark=*, magenta] coordinates {(1.7, 1/1.4)};
 \addplot[mark=o, magenta] coordinates {(1.7, 3)};
\end{axis}
\end{tikzpicture}
  \captionof{figure}{A case in which condition \eqref{Lequal1} does not hold, as $L_{x_0,x_0}<1$.}
 % \label{fig:test2}
\end{subfigure}
\caption{An illustration of two possible behaviours of the function $L_{x_0,x_0}$ in a neighbourhood of the Malthus exponent. The crucial difference between \eqref{Lequal1}, \eqref{Lbiggerthan1} and \eqref{Lfinitederivative} is the behaviour of the function as it crosses the dashed line at level one.}
\end{figure}

\subsection{A remarkable martingale}
Assume that \eqref{Lequal1} holds. Fix $x_0\geq0$ and define the function
\begin{equation}
    \h (x) \coloneqq L_{x,x_0}(\lambda), \quad \quad x \geq 0.
\end{equation}

\begin{lem}
\label{Lemmahcontinuous}
The function $\h$ is continuous in $\zeroinfc$.
\end{lem}
\begin{proof}
The function $\h$ is strictly positive in $[0, \infty)$, thanks to \ref{Sec2irreducible} and \ref{Sec2entranceboundary}. The argument in Corollary $4.3$ in \cite{BW18} ensures that $\h$ is continuous in $(0, \infty)$. Finally, we show that $\h$ is also continuous at $0$. Indeed, the Markov property entails that, for all $\varepsilon < x_0$,
\begin{equation}
\h(0) = L_{0, \epsilon}(\lambda) \h(\varepsilon).
\end{equation}
Next, we observe that 
\begin{equation}
\lim_{\varepsilon \to 0+} e^{- \lambda H(\varepsilon)} \Eg_{H(\varepsilon)} \mathbf{1}_{ \{H(\varepsilon) < \infty \}} = 1,
\end{equation}
and so, by Fatou's Lemma,
\begin{equation}
1 \leq \liminf_{\varepsilon \to 0+} \E_0 \left[ e^{- \lambda H(\varepsilon)} \Eg_{H(\varepsilon)} \mathbf{1}_{ \{H(\varepsilon) < \infty \}} \right] = \liminf_{\varepsilon \to 0+} L_{0, \epsilon}(\lambda).
\end{equation}
Let $\Lambda_{\epsilon}$ the event that the deterministic flow starting from $0$ reaches $\varepsilon$ without making any jump. Since the flow is strictly increasing, the jumps are only negative and the total jump rate is bounded, we have that $p_{\varepsilon} \coloneqq \p_0(\Lambda_\varepsilon) \uparrow 1$ when $\varepsilon \to 0^+$. Under this event, the hitting time is deterministic, say $H(\varepsilon) = s (0,\varepsilon)$, with $s(0,\epsilon) \downarrow 0$ when $\varepsilon \to 0^+$. We introduce now a geometric random variable $G(\varepsilon)$, with parameter $p_{\varepsilon}$. The number of jumps of $X$ before reaching $\varepsilon$ is stochastically dominated by $G(\varepsilon)$. Hence, 
\begin{equation}
\Hit \varepsilon \leq s(0,\varepsilon) G(\varepsilon).
\end{equation}
Thus, $\E_0 \left( e^{-\lambda \Hit { \varepsilon}} \mc E_{\Hit { \varepsilon}} , \Hit { \varepsilon} < \infty \right) \leq \E \left( e^{ \delta  G( \varepsilon)} \right)$,
where $$\delta \coloneqq \left( \sup_{x \in [0, \varepsilon)} B(x)(N(x)-1) - \lambda \right)s(0,\varepsilon).$$
All that is left to check is that $G( \varepsilon)$ has finite exponential moment with exponent $\delta$. We know that
\begin{align}
\E \left[ e^{\delta G( \varepsilon)}\right]& = \frac{p_ \varepsilon}{1-p_ \varepsilon} \sum_{k \geq 1} \left( e^{\delta} (1-p_ \varepsilon)\right)^k, 
\end{align}
which is finite if and only if $e^{\delta} (1-p_ \varepsilon) < 1$, \textit{i.e.}, $\delta < - \log (1-p_ \varepsilon)$. This reads, 
\begin{equation}
\sup_{x \in [0, \varepsilon)} B(x)(N(x)-1) < - \frac{\log (1-p_ \varepsilon)}{s(0,\varepsilon)} + \lambda,
\end{equation}
which is clearly true for $\epsilon$ small enough. In this case,
\begin{align}
\E \left[ e^{\delta G( \varepsilon)}\right]& = \frac{1}{1 - e^{\delta} (1-p_ \varepsilon)} \; \to \; 1, 
\end{align}
when $\epsilon \to 0$. This implies that 
\begin{equation}
\lim_{\varepsilon \to 0+} L_{0, \epsilon}(\lambda) =1,
\end{equation}
proving the claim.
\end{proof}

\begin{lem}
\label{LemmaMinageneralsetting}
Assume \ref{Sec2upwardskip}-\ref{Sec2positiverettimes}. If \eqref{Lequal1} holds, the process
\begin{equation}
    \label{DefinitionofthemartingaleMin the generalsettinng}
    \mc M_t = e^{- \lambda t} \h(X_t) \Eg_t, \quad \quad t \geq 0,
\end{equation}
is a $\p_x$-martingale for every $x\geq0$  with respect to the natural filtration $(\mc F_t)_{t \geq 0}$ of $X$.
\end{lem}

\begin{proof}
For $x>0$, we apply Theorem $4.4$ in \cite{BW18}. To show that it is also a martingale with respect to $\p_0$, we define the random variables $R_0 = 0 < R_1 \coloneqq H(x_0) < R_2 < \dots$ to be the return times to $x_0$, where $x_0$ is the point that appears in the definition of $\h$. The stopped process $(M_{t \wedge R_n})_{t \geq 0}$, is then a martingale and with an argument similar to the one in the proof of Theorem $4.4$ in \cite{BW18}, we can take the limit $n \to \infty$ and obtain the statement. 
\end{proof}

The next step consists in using the martingale $\mc M$ to ``tilt'' the probability measure $\p_x$. In other words, we introduce the probability measure $\pt_x$ (and corresponding expectation $\Et$) defined by
\begin{equation}
\label{tiltedproba}
\tilde{\p}_x(A)= \E_x[\mathbf{1}_A \mc M_t ], \quad \quad \forall A \in \mc F_t.
\end{equation}
Since $\p_x$ is a probability measure on the space of càdlàg paths, the same holds for $\pt_x$. Let $Y=(Y_t)_{t \geq 0}$ be the process with distribution $\pt_x$. The finite-dimensional distributions of $Y$ are thus given in the following way. Let $0 \leq t_1 < \dots < t_n \leq t$, and $F: \mathbb{R}^n \to \mathbb{R_+}$. Then, 
\begin{equation}
\Et_x [F(Y_{t_1}, \dots , Y_{t_n} )] = \E_x[\mc M_t F(X_{t_1}, \dots , X_{t_n} ) ], \quad \quad x \geq 0.
\end{equation}

\begin{lem}
\label{LemmaYinageneralsetting}
Assume \ref{Sec2upwardskip}-\ref{Sec2positiverettimes} and \eqref{Lequal1}. Then the following hold.
\begin{enumerate}[label=(\roman*)]
\item $Y$ is a Markov process, recurrent in $\zeroinf$. Moreover, denoting by $H_Y(x) = \inf \{ t>0 \; : \; Y_t= x \}$ the first hitting time of $x \geq 0$, one has that for all $x>0$,
\begin{equation}
    \Et_x\left( H_Y(x) \right) = - L'_{x,x}(\lambda).
\end{equation}
As a result, $Y$ is positive recurrent if and only if \eqref{Lfinitederivative} holds. 
\item If the stronger \eqref{Lbiggerthan1} holds, then $Y$ is exponentially recurrent, which means that it exists $\epsilon > 0$ such that  $\Et_{x} \left[\exp(\epsilon H_Y(x) ) \right] < \infty$.
\end{enumerate}
\end{lem}

\begin{proof}
\begin{enumerate}[label=(\roman*)]
    \item The process $Y$ is Markov because $\mc M$ is multiplicative and the (strong) Markov property is preserved by transformations based on multiplicative functionals. We denote $\pt_x$ the law of $Y$ started at $x\geq0$ and $\Et_x$ the corresponding expectation. To show that $Y$ is recurrent, we observe that, for $x>0$, 
    \begin{align}
\Et_{x} \left[H_Y(x)  < \infty \right] & = \lim_{t \to \infty}  \Et_{x} \left[ H_Y(x) \leq t \right] = \lim_{t \to \infty}  \E_{x} \left[\mc M_{t}, H(x) \leq t \right] \\ & = \lim_{t \to \infty}  \E_{x} \left[\mc M_{H(x)}, H(x) \leq t\right] \\
& =\E_{x} \left[\mc M_{H(x)}, H(x) < \infty \right] = L_{x,x} (\lambda)= 1,
\end{align}
where the second inequality comes from the definition of probability tilting, the third from the optional sampling theorem and the last from the monotone convergence theorem.
\begin{comment}Since $Y$ is recurrent, it is well-known that its stationary measure $m_0$ is given by the occupation measure, defined as
\begin{equation}
\langle m_0, f \rangle  \coloneqq \pt_{x_0} \left( \int_0^{H_Y(x_0)} f(Y_s) ds \right), \quad \quad f \in \mc C_c.
\end{equation}
Hence, t\end{comment} 
To show that it is actually positive recurrent, we note that, for $x>0$,
\begin{align}
\label{computaitontotalmassoftheoccupationmeasure}
\Et_{x} \left[H_Y(x)  \right] & = \int_0^{\infty} \left[ 1- \Et_{x} \left[H_Y(x) \leq t \right]  \right] dt \\
& =\int_0^{\infty} \left[ 1- \E_{x} \left[\mc M_t, H(x) \leq t \right] \right] dt \\
& =\int_0^{\infty}  \E_{x} \left[\mc M_{H(x)}, t < H(x) < \infty \right]  dt \\
& = \E_{x} \left[H(x) \mc M_{H(x)}, H(x) < \infty \right] =  - L'_{x,x}(\lambda),
\end{align}
which proves the assertion. 
\item With similar computations as above, one can prove that, since $H_Y(x) < \infty$ a.s., we have, 
\begin{align}
 \Et_{x} \left[\exp(\epsilon H_Y(x) ) \right] & = \lim_{t \to \infty}  \Et_{x} \left[\exp(\epsilon H_Y(x) ), H_Y(x) < t \right] \\
 & = \lim_{t \to \infty}  \E_{x} \left[\mc M_{H(x)} \exp(\epsilon H(x) ), H(x) < t \right] \\
 &  =\E_{x} \left[ \mc E_{H(x)} \exp( (\epsilon- \lambda) H(x) ), H(x) < \infty \right] = L_{x,x}(\lambda - \epsilon),
\end{align}
which is finite for $\epsilon$ small enough thanks to condition \eqref{Lbiggerthan1}, when $x=x_0$. The case of a general $x>0$ follows easily. \qedhere
\end{enumerate} 
\end{proof}

\subsection{A family of supermartingales}
From the previous lemma, it is clear that condition $L_{x_0, x_0}(\lambda) =1$ is necessary to construct the martingale $\mc M$. The next result shows that, when $q$ is such that $L_{x_0, x_0}(q)<1$, \textit{i.e.} $q \geq \lambda$, we can associate to $X$ a family of supermartingales. 
Fix $x_0\geq0$ and $q \leq \lambda$ and define the function
\begin{equation}
\label{definitionofhq}
    \hq (x) \coloneqq L_{x,x_0}(q), \quad \quad x \geq 0.
\end{equation}
Adapting the proof of Theorem $4.4$ in  \cite{BW18}, we have the following Lemma. 
\begin{lem}
\label{LemmaSinageneralsetting}
Let $q \geq \lambda$. The process
\begin{equation}
    \label{DefinitionofthesupermartingaleSinthegeneralsetting}
    \mc S_t^{(q)} \coloneqq e^{-qt} \h_q(X_t) \Eg_t, \quad \quad t \geq 0
\end{equation}
is a $\p_x$-supermartingale for every $x\geq0$ with respect to the natural filtration $(\mc F_t)_{t \geq 0}$ of $X$.
\end{lem}

As before, we use the supermartingale $\mc S^{(q)}$ to ``tilt''  the probability measure $\p_x$ and introduce a family of possibly defective (\text{i.e.} possibly with a finite lifetime $\zeta$) Markov processes $Y^{(q)} = \left(Y_t^{(q)}\right)_{0 \leq t < \zeta}$. More precisely, the distribution of the Markov process $Y^{(q)} = \left( \Ytq \right)_{0 \leq t < \zeta}$, that we denote by $\p^{(q)}$, is defined in the following way. For $t \geq 0$ and every non-negative functional $F$ defined on Skorokhod's space $\mc D_{[0,t]}$ of càdlàg paths $\omega : [0,t] \to (0, \infty)$, 
\begin{equation}
\label{definitionofYq}
\Eq_x [F( (\Ytq)_{0 \leq s \leq t}), \zeta > t] = \frac{1}{\h_{q}(x)} \E_x [ \Stq F((X_s)_{0\leq s \leq t})], \quad \quad x>0.
\end{equation}

\begin{lem} 
\label{LemmaYqinageneralsetting}
If there exists a $r \geq \lambda$ such that $Y^{(r)}$ is point-recurrent in $\zeroinf$, then $\mc S_t^{(r)}$ is a martingale, and $L_{x,x}(r) = 1$ for every $x>0$. 
\end{lem}

\begin{proof}
If $Y^{(r)}$ is positive recurrent, it cannot be defective, \textit{i.e.}, $\p^{(r)}_x \left( \zeta = \infty \right) = 1$. This is equivalent to say that $\E^{(r)}_x \left( \mc S_t^{(r)} \right) = \h_{r}(x)$, which implies that $\mc S^{(r)}$ is a martingale for every $x >0$. Since  $Y^{(r)}$ is point-recurrent, we have that, for every $x>0$, 
    \begin{align}
        1 & = \lim_{t \to \infty} \p_x^{(r)} \left(H_{_{Y^{(r)}}}(x) \leq t \right) = \lim_{t \to \infty} \frac{1}{\h_{r}(x)} \E_x \left[  \mc S_t^{(r)}, H(x) \leq t \right] \\
        & =  \lim_{t \to \infty} \frac{1}{\h_{r}(x)} \E_x \left[  \mc S_{H(x)}^{(r)}, H(x) \leq t \right] = \lim_{t \to \infty} \E_x \left[  e^{- r H(x)} \mc E_{H(x)}, H(x) \leq t \right] \\
        & = \E_x \left[  e^{- r H(x)} \mc E_{H(x)}, H(x) < \infty \right] = L_{x,x}(r),
    \end{align}
where the equalities follow from the optional sampling theorem and the monotone convergence theorem. \qedhere
\end{proof}

\subsection{The process killed when exiting compact sets}
In this paragraph, we focus on the behaviour of the process $X$ killed when exiting compact sets. A necessary preamble for the rest of the analysis is the irreducibility. In fact,  even though $X$ is irreducible in the positive half-line by \ref{Sec2irreducible}, it may happen that there exist some $0<a<b$ such that the process is no longer irreducible, when killed exiting $[a,b]$. We define the \textit{first exit-time} from $[a,b]$ 
\begin{equation}
\label{definitionexittime}
\sigma(a,b) \coloneqq \inf \{ t > 0 \; : \; X_t \notin [a,b] \},
\end{equation}
and we call an interval $(a,b)$ \textit{good} if the process killed at time $\sab$ remains irreducible, \textit{i.e.},
\begin{equation}
\p_x(H(y) < \sab) > 0 \quad \quad  \text{for all} \; x,y \in (a,b).
\end{equation}

The argument of Lemma $3.1$ in \cite{BERT18} shows the following.
\begin{lem}
\label{lemmagoodintervals}
Assume that \ref{Sec2upwardskip}-\ref{Sec2positiverettimes} hold. Then, for every $\epsilon \in (0,1)$,  there exists a good interval $(a,b)$ with $a < \epsilon$ and $b > 1/\epsilon$.
\end{lem}

The next step consists in applying the Krein-Rutman theorem. We consider the Banach space $\Chatab$ of continuous functions $f : [a,b] \to \mathbb{R}$ with $f(b)=0$ endowed with the supremum norm $\norm{f} = \sup_{x \in [a,b)} |f(x)|$. We assume $f(b)=0$ because the process started at $b$ leaves $[a,b]$ immediately\footnote{In this case $b$ is said to be an \textit{exit boundary}.}. We do not assume yet that $[a,b]$ is a good interval, but this will be crucial at a later point.

Recalling that $\norm{g}_{\infty}<\infty$, we define $\qg \coloneqq 1 + \norm{g}_{\infty}$, so that 
\begin{equation}
\mc E_ t e^{-t \qg} \leq e^{-t} \quad \quad \text{for all} \; t \geq 0.
\end{equation}
Then, we introduce the operator 
\begin{equation}
\Uab f(x) \coloneqq \E_x \left( \int_0^{\sab} f(X_t)  \Eg_ t e^{-t \qg} dt   \right), \quad \quad x \in [a,b],
\end{equation}
that is defined for every bounded measurable function $f : [a,b] \to \mathbb{R}$. The operator $U_{a,b}$ maps $\Chatab$ into itself. The family of functions $\{\Uab f \; : \; \norm f \leq 1 \}$ is equicontinuous; the proof is similar to that of Lemma $3.2$ in \cite{BERT18} and we leave the details to the reader.
\begin{comment}
The subspace $\Chatabp$ of non-negative functions is $\Chatab$ is a reproducing cone, \text{i.e.}, a closed convex set which is stable by multiplication by non-negative constants, and a function $f \in \Chatab$ can be decomposed as $f= f^+ - f^-$. The operator $\Uab$ is clearly positive, which means that it maps $\Chatabp$ into itself. The dual space of $\Chatab$, that we denote $\Chatab^*$, is given by the set of finite Borel measures on $[a,b]$ that have no atoms at $b$. The dual operator $\Uabd$ is an operator on  $\Chatab^*$ such that, for any $m \in \Chatab^*$, it holds
\begin{equation}
\langle \Uabd m, f \rangle = \langle m, \Uab f \rangle \quad \quad f \in \Chatabp.
\end{equation}
\end{comment}
$\Uab$ satisfies the hypothesis of the Krein-Rutman theorem (see for example the requirements of Theorem $9.5$ in Deimling \cite{DEIMLING85}), which entails the following result.

\begin{prop}
\label{kreinrutman}
Let $(a, b)$ a good interval. Then
\begin{enumerate}[label=(\roman*)]
\item the spectral radius $r(a,b)$ of $\Uab$ is positive,
\item there exist a function $\hab \in \Chatab$ strictly positive and a finite Borel measure $\nuab$ on $[a,b]$ with no atoms at $b$ such that 
\begin{equation}
\Uab \hab = r(a,b) \hab, \quad \Uabd \nuab = r(a,b) \nuab, \quad \text{and} \quad  \langle \nuab, \hab \rangle = 1,
\end{equation}
\item the spectral gap holds, i.e., if $r  \neq r(a,b)$ belongs to the spectrum of $\Uab$, then $|r| < r(a,b)$.
\end{enumerate}
\end{prop}

Thanks to Proposition \ref{kreinrutman}, we can introduce the quantity 
\begin{equation}
\roab \coloneqq \qg - \frac{1}{r(a,b)},
\end{equation}
and we have the following result. The proof follows adapting the one of Lemma $3.4$ in \cite{BERT18}.

\begin{lem}
The process
\begin{equation}
\Mab (t) \coloneqq \mathbf{1}_{\{t < \sab \}} \hab  \Eg_t e^{-t \roab}, \quad t \geq 0
\end{equation}
is a $\p_x$- martingale for every $x \in [a,b]$.
\end{lem}

\section{Characterisation of the Malthusian behaviour}
\label{Section3}

In this section we prove some first important results. The first goal is to establish existence and uniqueness of the growth-fragmentation semigroup and to derive the Feynman-Kac representation \eqref{FKrep}. The second one is to prove Theorem \ref{Theorem1}, that gives necessary and sufficient conditions for the Malthusian behaviour \eqref{convergenceofthesemigroup}, in terms of the Laplace transform $L$ and the Malthus exponent $\lambda$ defined respectively in \eqref{Lxy} and \eqref{definitionoflambda}.

We start by introducing some notation. For $x \geq 0$ and $t \geq 0$, we denote by $\phi(x,t)$, the flow given by the solution to the differential equation
\begin{equation}
\label{definitiondiffequationtau}
\begin{cases}
\text{d}\phi(t, x) = \tau(\phi(t,x))dt, \\
\phi(0,x)= x,
\end{cases}
\end{equation}
that exists and is unique for all $x \geq 0$ thanks to \eqref{tauLipschitz}. For $0 \leq x<y$, we denote by $s(x,y)$ the time that $\phi$ needs to travel from $x$ to $y$, namely
\begin{equation}
\phi(s(x,y), x)=y.
\end{equation}
Note that $s(\cdot, \cdot)$ is decreasing in the first variable and increasing in the second one. 
\begin{oss}
\label{remarkentrancenadnaturalboundary}
There is the explicit expression
\begin{align}
    s(x, y) = \int_x^y \frac{1}{\tau(z)} dz.
\end{align}
Comparing this to \eqref{definitionofT}, we see that $T= s(0,1)$. When \eqref{Tfinite} holds, the solution with initial condition $x(0)=0$ can \textit{enter from $0$ in finite time}. On the contrary, when \eqref{Tfinite} fails (for example in the cases analysed in \cite{BERT18, BW18, BC19}), the solution to \eqref{definitiondiffequationtau} with initial condition $x(0)=0$ is $\phi(t, 0)=0$ for all $t \geq0$. On the other hand, \eqref{Timetoinfinity} ensures that the solution cannot explode in finite time.
\end{oss}

\begin{comment}
\textcolor{magenta}{to move somewhere else}
For $0 \leq x<y$, let $\Lambda_{x,y}$ be the event that the process $X$ started at $x$ reaches $y$ before making any jump. Since the total jump rate when the process is located at $z >0$ is $B(z)N(z)$, we have
\begin{align}
p(x,y) \coloneqq \p_x (\Lambda_{x,y}) & = \exp \left( - \int_0^{s(x,y)} B(\phi(t,x))N(\phi(t,x)) dt \right) \\ & = \exp \left( -  \int_x^y \frac{B(z)N(z)}{\tau(z)} dz \right),
\end{align}
which is increasing in $x$ and decreasing in $y$.
In particular, 
\begin{equation}
    \inf_{x \in (0, y)} p(x, y)= \begin{cases}
        p(0, x) \quad \quad  \; T< \infty \\
        0 \quad \quad \quad \quad \; \; T = \infty.
    \end{cases}
\end{equation}
\end{comment}

\subsection{Existence and uniqueness of the semigroup}
The goal of this subsection is to prove the existence and uniqueness of a semigroup generated by $\mc A$.
Let $\Chat$ denote the Banach space of continuous functions $f: [0, \infty) \to \mb R$ vanishing at infinity, endowed with the supremum norm $\norm{\cdot}_{\infty}$. We view the growth-fragmentation operator $\mc A$, defined in \eqref{definitiongfoperator}, as an operator on $\Chat$. Its domain $\mc D (\A)$ contains the space of functions $f \in \Chat$ such that $\tau f' \in \Chat$. 

\begin{comment}
The weak form of the growth-fragmentation equation is ,
\begin{equation}
\partial_t \langle u_t, f \rangle = \langle u_t, \A f \rangle
\end{equation}
and we would like to express the solutions as
\begin{equation}
\label{weaksolutiongfe}
\langle u_t, f \rangle = \langle u_0, T_t f \rangle,
\end{equation}
where $(T_t)_{t \geq 0}$ is a semigroup on a certain Banach space of functions on $(0, \infty)$ that has to be determined and whose infinitesimal generator extends $\A$. 
\end{comment}

We also assume the following technical assumption on the fragmentation kernel: for all compact sets $E \subset [0, \infty)$,
\begin{equation}
\label{vagueconvergenceofthekernel}
    \lim_{x \to \infty} B(x) \int_E k(x, y) dy = 0,
\end{equation}
\textit{i.e.}, the rate at which a particle with size $x >0$ produces particles whose sizes are in $E$ tends to $0$ as $x \to \infty$.

\begin{lem}
\label{existenceanduniquenesssemigroupTfinite}
Assume \eqref{tauLipschitz}, \eqref{Bcontinuousbounded}, \eqref{kcontinuous}, \eqref{Nconitnuousbounded}, \eqref{Timetoinfinity}, \eqref{conditionequivalenttoirreducibility} and \eqref{vagueconvergenceofthekernel}. Then, there exists a unique positive strongly continuous semigroup on $\Chat$ whose infinitesimal generator coincides with $\A$ in the space of differentiable functions vanishing at infinity such that $\tau f' \in \Chat$. 
\end{lem}
\begin{proof}
From \eqref{equationnumberofchildren}, for $x \geq 0 $, $\mc A$ can be written as 
\begin{align}
\A f (x) = \tau(x)f'(x) + B(x)  \int_0^x \left(f(y)- f(x) \right) k(x,y) dy + B(x) \left(N(x) - 1 \right)  f(x) .
\end{align}
We introduce the operator $\At f \coloneqq \A f - \norm{\bn}_{\infty} f$, defined on $\mc D (\At) = \mc D( \mc A)$. Plainly, $\norm{\bn}_{\infty}  - B(N-1) \geq 0$. 
Note that, if one shows that $\tilde{A}$ generates a unique strongly continuous contraction semigroup $(\tilde{T}_t)_{t \geq 0}$ on $\Chat$, then $T_t f \coloneqq \exp (t  \norm{\bn}_{\infty}) \tilde{T}_t  f$ is a positive strongly continuous semigroup on $\Chat$ with infinitesimal generator $\mc A$.

Conversely, let $(T_t)_{t \geq 0}$ be a positive strongly continuous semigroup on $\Chat$ with infinitesimal generator $\mc A$. 
Then, $\tilde{T}_t \coloneqq \exp (- t  \norm{\bn}_{\infty})T_t$ defines a strongly continuous contraction semigroup with infinitesimal generator $\tilde{A}$, since 
\begin{equation}
\norm{\tilde{T}_t f}_{\infty} \leq \exp\{- \norm{\bn}_{\infty}t\} \norm{f}_{\infty}.
\end{equation}
From existence and uniqueness of the semigroup generated by $\tilde{A}$, we will get the uniqueness of $(T_t)_{t \geq 0}$.

To show that the semigroup $(\tilde{T}_t)_{t \geq 0}$ exists, we construct a Markov process, say $(Z_t)_{t \geq 0}$, having generator $\At$ on $\mc D (\At)$. The evolution of $Z$ starting from $x \geq 0$ is the following. 
Consider the functions
\begin{equation}
F(t, x) \coloneqq \exp \left( - \int_0^t B(\phi(s, x))N(\phi(s, x) ds \right) = \exp \left( - \int_x^{\phi(t, x)} \frac{B(z)N(z)}{\tau(z)} dz \right)
\end{equation}
and
\begin{align}
    G(t, x) & \coloneqq 
    \exp \left( - \int_0^t \left(\NormB - B(\phi(s, x))(N-1)(\phi(s, x)) \right) ds \right) \\
    & = \exp \left( - \int_x^{\phi(t, x)} \frac{\left( \NormB - B(z)(N-1)(z) \right)}{\tau(z)} dz \right).
\end{align}
Now select two independent random variables $K_1$ and $S_1$ such that $\p (K_1 > t) = F(t,x)$ and $\p(S_1 > t) = G(t, x)$. Consider also a random variable $P_1$ independent from the others, with distribution 
\begin{equation}
\label{definitionprobaposiitonafterjump}
    \frac{k(\phi(K_1, x), y)}{N(\phi(K_1, x))} dy.
\end{equation}

Let $T_1 = K_1 \wedge S_1$. On the event $T_1 = S_1$, the process is killed, \textit{i.e.,}
\begin{equation}
    Z_t = \begin{cases}
        \phi(t, x) \quad \quad t < T_1 \\
        \partial  \quad \quad \quad \quad \; t \geq T_1.
    \end{cases}
\end{equation}
On the event  $T_1 = K_1$, the trajectory of $Z$ starting from $x \geq 0$ is given by 
\begin{equation}
    Z_t = \begin{cases}
        \phi(t, x) \quad \quad \; t < T_1 \\
        P_1  \quad \quad \quad \quad t = T_1,
    \end{cases}
\end{equation}
and then the dynamics starts afresh from $P_1$. More precisely, we select two independent random variables $K_2$ and $S_2$ such that $\p (K_2> t) = F(t,P_1)$ and $\p(S_2 > t) = G(t, P_1)$. Consider also a random variable $P_2$ independent from the others, with distribution 
\begin{equation}
    \frac{k(\phi(K_2, P_1), y)}{N(\phi(K_2, P_1))} dy.
\end{equation}
Then we define $T_2 = K_2 \wedge S_2$ and, again, if $T_2= S_2$, the process is killed and, otherwise, the dynamics continues in a similar way.
Let $m=\inf \{i \; | \; T_i = S_i \}$. Then we have a piecewise deterministic trajectory $Z_t$ with jump times $T_1, T_1 + T_2, \dotsc,  \sum_{i=1}^{m}T_{i}$ killed at time $K=  \sum_{i=1}^{m}T_{i}$. By construction, $Z$ is a Markov process and it has generator $\At$ on $\mc D (\At)$ (see for example \cite{MD86}).

Uniqueness of the semigroup follows applying Theorem $4.1$ in chapter $4$ of \cite{EK86}, with $\mc A'$ being $\tilde{\mc A}$. Clearly, the set $\overline{\mc D(\mc A')} = \Chat$ is separating and we just need to verify $\overline{\mc R (\lambda - \mc A')}= \overline{\mc D(\mc A')}$. This follows from Lemma $4.2$ and Theorem $4.3$ in chapter $4$ of \cite{EK86}.

We thus proved that there exists a unique positive strongly continuous semigroup $(\tilde{T}_t)_{t \geq 0}$ that has infinitesimal generator $\tilde{A}$, which implies the statement.
\end{proof}

\subsection{A Feynman-Kac representation and Malthusian behaviour}

In this subsection, we establish the Feynman-Kac representation \eqref{FKrep}. First, the same argument used in the proof of Lemma \ref{existenceanduniquenesssemigroupTfinite} shows that the operator $\G$ defined in \eqref{definitionofg}, with domain $\mc D( \mc G) = \mc D( \mc A)$, generates a strongly continuous contraction semigroup on $\Chat$ and, hence, it is the generator of a conservative Feller Markov process $X$ on $[0, \infty)$. From the expression of $\mc G$, we see that $X$ belongs to the class of \textit{piecewise-deterministic Markov processes} introduced by Davis \cite{D84}. Under $\p_x$, any path $t \mapsto X_t$ follows the deterministic flow $\phi(t,x)$ defined in \eqref{definitiondiffequationtau} up to a random time at which it makes its first (random) jump. When the jump occurs, the position after it, say $y$, is chosen accordingly to \eqref{definitionprobaposiitonafterjump} and the dynamics starts afresh from $y$. Note that $X$ has only negative jumps, as it is clear from the definition of the jump kernel \eqref{definitionprobaposiitonafterjump}. Moreover, \eqref{Bcontinuousbounded} and \eqref{Nconitnuousbounded} ensure that the jumps of $X$ never accumulate.

Note further that, by \eqref{Timetoinfinity}, the process cannot reach $\infty$ in finite time. On the contrary, by \eqref{Tfinite}, $0$ is an \textit{entrance boundary} for $X$. As stated in the Introduction, we assume that $X$ is irreducible in $\zeroinf$. Since $\tau$ is positive and $X$ has only negative jumps, this is equivalent to \eqref{conditionequivalenttoirreducibility}. For the proof, we refer to Lemma $3.1$ in \cite{BERT18}.

\begin{oss}
\label{Remarkirreducibilityandentrancepoint}
We stress that $X$ is not irreducible in $[0, \infty)$. In fact, the process started at $0$ can hit any target point $y >0$ with positive probability, but the process started at $x >0$ does not hit $0$ in finite time, due to the fact that the the probability that $X$ hits $0$ by a jump is zero together with the fact that the jumps never accumulate and that condition \eqref{Tfinite} holds.
\end{oss}

To sum up, $X$ satisfies the properties \ref{Sec2upwardskip}-\ref{Sec2positiverettimes}. Moreover, $B(x)(N(x)-1)$ is continuous and bounded and so, we can rely on the results presented in Section \ref{Section2}, with the choice $g(x)= B(x) \left( N(x) -1 \right)$.  Notice that, in this case, the functional  $\left( \Eg_t \right)_{t \geq 0}$ is exactly the functional $\left( \mc E_t \right)_{t \geq 0}$ defined in \eqref{definitionofE}. The next result is the Feynman-Kac representation of the semigroup.

\begin{lem}
\label{lemmaFK}
The growth-fragmentation semigroup can be expressed in the form
\begin{equation}
T_t f (x) = \E_x \left[ \mc E_t f(X_t) \right].
\end{equation}
\end{lem}
\begin{proof}
Since $\mc G$ is the generator of $X$, from Dynkin's formula, for every $f \in \mc D(\mc G))$,
\begin{equation}
    f(X_t) - \int_0^t \mc G f(X_s) ds, \quad t \geq 0
\end{equation}
is a $\p_x$-martingale for every $x \geq 0$. In addition, $(\mc E_t)_{t \geq 0}$ is a stochastic process with bounded variation and $d\mc E_t = B(X_t)\left(N(X_t)-1 \right)\mc E_t dt$. Thus, it follows from the integration by parts formula for stochastic calculus, that 
\begin{equation}
    \mc E_t f(X_t) - \int_0^t \mc E_s \mc G f(X_s) ds - \int_0^t B(X_s)\left(N(X_s)-1 \right) \mc E_s f(X_s) ds = \mc E_t f(X_t) - \int_0^t \mc E_s \mc A f(X_s) ds
\end{equation}
is a local martingale. Since this local martingale remains bounded on any finite time interval, it is a true martingale (see Theorem $I.51$ in \cite{PROTTER05}). Taking expectations and using Fubini's theorem, we conclude that
\begin{equation}
    \E_x \left( \mc E_t f(X_t) \right) - f(x) = \int_0^t \E_x \left( \mc E_s \mc A f(X_s) \right) ds, 
\end{equation}
which means that $\mc A$ is the generator of the semigroup $\E_x \left( \mc E_t f(X_t) \right)$.
By uniqueness, we get the Feynman-Kac representation.
\end{proof}

We are now ready to state an important result concerning the Malthusian behaviour. To this end, we recall the definition of the function $L_{x,y}$  and the Malthus exponent $\lambda$ introduced respectively in \eqref{Lxy} and \eqref{definitionoflambda}. The following theorem provides necessary and sufficient conditions in terms of the function $L_{x,y}$ for the convergence of $e^{-\lambda t} T_t$ to an asymptotic profile. Moreover, it gives an explicit expression of the latter. 
\begin{thm}
\label{Theorem1}
Assume \eqref{tauLipschitz}, \eqref{Bcontinuousbounded}, \eqref{kcontinuous}, \eqref{Nconitnuousbounded}, \eqref{Timetoinfinity}, \eqref{conditionequivalenttoirreducibility} and  \eqref{vagueconvergenceofthekernel}. 
\begin{enumerate}[label=(\roman*)]
\item \label{Theorem1if}  
If
\begin{equation}
\label{Lequal1andfinitederivative}
L_{x_0, x_0}(\lambda)=1 \quad \text{and} \quad -L'_{x_0, x_0}(\lambda) < \infty, 
\end{equation}
then, the Malthusian behaviour \eqref{convergenceofthesemigroup} holds with $\rho = \lambda$ and $\h$ and $\nu$ defined as in \eqref{definitionofh} and \eqref{definitionofnu}.
\item \label{Theorem1onlyif} Conversely, if \eqref{convergenceofthesemigroup} holds for smoothly compactly supported functions $f$, then \eqref{Lequal1andfinitederivative} holds for $\lambda = \rho$.
\end{enumerate}
\end{thm}

Fix $x_0>0$ and let $\h(x)$ be as in \eqref{definitionofh}. Under assumption \eqref{Lequal1andfinitederivative}, Lemma \ref{LemmaMinageneralsetting} ensures that the process 
\begin{equation}
\label{DefinitionofM}
\mc M_t= e^{-\lambda t}\mc E_t \frac{\h(X_t)}{\h(X_0)}, \quad \quad t \geq 0
\end{equation}
is a martingale under $\p_x$, $x \geq 0$. Thus, we can use it to tilt the probability measures associated to $X$ in order to obtain a recurrent Markov process $Y= (Y_t)_{t \geq 0}$. As in Section \ref{Section2}, we call $\pt_x$ (resp.\ $\Et_x$) the law (resp.\ the expectation) of the process $X$ condition to start at $X_0=x$, $x \geq 0$. Since $\pt$ is absolutely continuous with respect to $\p$, $Y$ inherits the properties \ref{Sec2upwardskip}-\ref{Sec2positiverettimes}.

\begin{proof}[Proof of Theorem 3.3(i)]
Combining \eqref{FKrep} and \eqref{DefinitionofM},
\begin{align}
\label{FCwithY}
T_t f (x) & = \E_x \left[ \mc E_t f(X_t) \right] = e^{\lambda t} \h(x) \E_x \left[ \mc E_t \frac{\h(X_t)}{\h(X_0)} \frac{f(X_t)}{\h(X_t)} \right] = e^{\lambda t} \h(x) \Et_x \left[\frac{f(Y_t)}{\h(Y_t)} \right].
\end{align}
Since \eqref{Lequal1andfinitederivative} holds, Lemma \ref{LemmaYinageneralsetting} shows that $Y$ is positive recurrent. By standard results, the stationary measure of a recurrent Markov process is given by its occupation measure normalized to be a probability measure. Moreover, since $Y$ is piecewise-deterministic and follows the deterministic flow $dy(t)= \tau(y(t)) dt$ between consecutive jumps, it can be proved that its occupation measure is absolutely continuous with respect to the Lebesgue measure, with a locally integrable and everywhere positive density that is
\begin{equation}
\label{densityofthesdofy}
\frac{q(x_0, y)}{\tau(y) q(y, x_0)}, \quad \quad y>0,
\end{equation}
where $q(x, y) \coloneqq \pt_x (H_Y(y) < H_Y(x))$. For the proof, we refer to Lemma $5.2$ in \cite{BW18}.
Combining \eqref{computaitontotalmassoftheoccupationmeasure}, \eqref{FCwithY} and \eqref{densityofthesdofy}, we can conclude that
\begin{equation}
\lim_{t \to \infty} e^{- \lambda t} T_t f(x) = h(x) \int_0^{\infty} \frac{f(y)}{h(y)} \times \frac{1}{\tau(y)|L'_{y,y}(\lambda)|} dy = h(x) \langle \nu, f \rangle,
\end{equation}
where $\nu(dx)$ is precisely the probability measure defined in \eqref{definitionofnu}.
\end{proof}

\begin{oss}
\label{remarkexpconvergence}
When \eqref{Lbiggerthan1} holds, Lemma \ref{LemmaYinageneralsetting} shows that $Y$ is exponentially recurrent. In this case, Kendall's renewal theorem ensures that the convergence takes place exponentially fast.
\end{oss}

The second part of Theorem \ref{Theorem1} states that \eqref{Lequal1andfinitederivative} is not only sufficient for the Malthusian behaviour \eqref{convergenceofthesemigroup}, but also necessary and, in particular, whenever \eqref{convergenceofthesemigroup} holds, the leading eigenvalue $\rho$ coincides with the Malthus exponent $\lambda$ defined in \eqref{Lequal1andfinitederivative}. We actually prove a stronger result.

\begin{lem}
\label{lemmaonlyifstrongerresult}
Assume \eqref{tauLipschitz}, \eqref{Bcontinuousbounded}, \eqref{Nconitnuousbounded}, \eqref{Timetoinfinity}, \eqref{conditionequivalenttoirreducibility} and \eqref{vagueconvergenceofthekernel}. Suppose that for some $\alpha \in \mathbb{R}$:
\begin{enumerate}[label=(\roman*)]
\item \label{ass1onlyif} there exists $x_1 >0$ and a  continuous function $f: (0, \infty) \to \mathbb{R}_+$ with compact support and $f \not\equiv 0$, such that 
\begin{equation}
\limsup_{t \to \infty} e^{- \alpha t} T_t f(x_1) < \infty,
\end{equation}
\item  \label{ass2onlyif} there exist $x_2 > 0$ and a continuous function  $g: (0, \infty) \to \mathbb{R}_+$ with compact support such that
\begin{equation}
\liminf_{t \to \infty} e^{- \alpha t} T_t g(x_2) > 0.
\end{equation}
\end{enumerate}
Then $\alpha = \lambda$, \eqref{Lequal1andfinitederivative} holds and thus also the Malthusian behaviour \eqref{convergenceofthesemigroup} holds with $h$ and $\nu$ defined as in \eqref{definitionofh} and \eqref{definitionofnu}.
\end{lem}
The argument for proving this result belongs to the same vein as in the proof of \ref{Theorem1if}, with the difference that the role of the martingale $\mc M$ is now played by a family of supermartingales.

By contradiction with Proposition $3.3$ in \cite{BW18}, we have the following result. We refer to the proof of Lemma $2.4$ in \cite{BERT18} for a more extensive argument.
\begin{lem}
Suppose that the assumption \ref{ass1onlyif} of Lemma \ref{lemmaonlyifstrongerresult} holds. Then $\alpha \geq \lambda$.
\end{lem}

Hence, we can refer to \eqref{definitionofhq} and consider the function 
\begin{equation}
    \h_{\alpha}(x) \coloneqq L_{x,x_0}(\alpha), \quad \quad x \geq 0.
\end{equation}

\begin{comment}
This function is bounded away from $0$ and $\infty$ on every compact interval of $[0, \infty)$. Indeed, Lemma $2.2$ in \cite{BERT18} shows that it is bounded away from $0$ and $\infty$ on every compact interval of $\zeroinf$. Moreover, conditioning on the event $\Lambda_{x,a}$ for every $0 \leq x < a$, it follows straightforwardly that it is also bounded away from $0$ in $[0, a)$, $a >0$. 
\end{comment}

As in \eqref{DefinitionofthesupermartingaleSinthegeneralsetting}, we can define the process
\begin{equation}
\Sta \coloneqq  e^{- \alpha t} \h_{\alpha} (X_t) \mc E_t, \quad \quad t \geq 0
\end{equation}
which is a supermartingale with respect to $\p_x$, $x\geq0$ thanks to Lemma \ref{LemmaSinageneralsetting}. 
In the same way as in Section \ref{Section2}, we can use $\Sta$ to introduce a possibly defective càdlàg Markov process $Y^{(\alpha)} = \left( \Yta \right)_{0 \leq t < \zeta}$, with law $\p^{(\alpha)}$. Since the distribution of $(\Yta)_{0 \leq s \leq t})$ under $\p^{(\alpha)}_x(\cdot \; | \; t < \zeta)$ is absolutely continuous with respect to that of $(X_s)_{0 \leq s \leq t})$ under $\p_x$, then the process it is irreducible on $(0,\infty)$ and $0$ is an entrance boundary. 

In the following we denote $Y \coloneqq Y^{(\alpha)}$. In the next lemma, we show that the the process $Y$ is indeed positive recurrent. The proof follows adapting the ones of Lemma $2.4$ and Corollary $2.5$ in \cite{BERT18} and is left to the reader.

\begin{lem}
\label{lemmaYrhoisrecurrent}
Suppose that assumptions \ref{ass1onlyif} and \ref{ass2onlyif} hold. Then, the process $Y$ is point-recurrent and positive recurrent in $\zeroinf$, that is
\begin{align}
 \Ea_x \left(H_Y(y)\right) < \infty \quad \quad \text{for every} \quad x, y >0,
\end{align}
where $H_Y(y) \coloneqq \inf \{t \in (0, \zeta ) \; : \; Y_t=y \}$ is the hitting time of $y >0$ by the process $Y$, with the convention that $\inf \emptyset = \infty$.
\end{lem}

Finally, we are ready to prove the second part of Theorem \ref{Theorem1}.

\begin{proof}[Proof of Theorem 3.3(ii)]
By Lemma \ref{lemmaYrhoisrecurrent}, $Y$ cannot be defective, \textit{i.e.}, $\pa_x \left( \zeta = \infty \right) = 1$. This is equivalent to say that $\Ea_x \left( \Sta \right) = \h_{\alpha}(x)$, which implies that $\mc S^{(\alpha)}$ is a martingale for every $x \geq 0$.
   
Thanks to Lemma \ref{LemmaYqinageneralsetting}, we get that $L_{x,x}(\alpha) = 1$ for every $x\geq0$, \textit{i.e.}, condition \eqref{Lequal1} holds. From this, we see that the function $\h_{\alpha}$ coincides with the function $\h$ defined in the proof of Theorem \ref{Theorem1}\ref{Theorem1if}, the martingale $\mc S^{(\alpha)}$ coincides with $\mc M$ and the process $Y$ is the same as the one defined in the previous section. This implies that since $Y$ is recurrent, condition \eqref{Lfinitederivative} must be satisfied, proving the assertion.
\end{proof}

\section{Proof of the main results}
\label{Section4}
This section is devoted to the proofs of Theorem \ref{Theorem2} and  Theorem \ref{Theorem3}. 

\subsection{Proof of Theorem 1.1}
Remark \ref{remarkexpconvergence} shows that \eqref{Lbiggerthan1} is a sufficient condition for the Malthusian behaviour with exponential convergence \eqref{definitionexponentialconvergence}. Thus, the goal of this section is to show that \eqref{conditionlimsupatinfinity} implies \eqref{Lbiggerthan1}, \textit{i.e.}, that there exist $q \in \mathbb{R}$ and $ x \in \zeroinf$ such that
\begin{equation}
\E_x \left[ e^{-q H(x)} \mc E_{H(x)}, H(x) < \infty  \right] \in (1, \infty).
\end{equation}
This will be proven by decomposing the excursions of $X$ away from its (properly chosen) starting point at certain exit times from (properly chosen) compact sets. Thus, first of all, we will fix a compact interval $[a,b]$, with given $0<a<b$ and, following Section \ref{Section2}, we study the process killed when exiting $[a,b]$. The second step consists in fixing the upper-boundary point $b$ large enough and letting the lower-boundary point $a$ go to $0+$. In these first two steps, the expectations of proper functionals at exit times will be computed with the help of specific martingales and supermartingales. Finally, the statement of the theorem follows putting the previous results together. 

We start by fixing a good interval $(a,b)$ and, following Section \ref{Section2}, we define $\qbn \coloneqq 1 + \NormB$ and introduce the operator on $\Chatab$
\begin{equation}
\Uab f(x) \coloneqq \E_x \left( \int_0^{\sab} f(X_t) \mc E_ t e^{-t \qbn} dt   \right), \quad \quad x \in [a,b].
\end{equation}

By Proposition \ref{kreinrutman}, we have that
\begin{equation}
\Uab \hab = r(a,b) \hab, \quad \Uabd \nuab = r(a,b) \nuab, \quad \text{and} \quad  \langle \nuab, \hab \rangle = 1,
\end{equation}
where $r(a,b) >0$ is the spectral radius, $\hab \in \Chatabp$ is strictly positive $\nuab$ is a finite Borel measure  on $[a,b]$ with no atoms at $b$ .
\begin{lem}
\label{expectationsinsideacompact}
Take any good interval $(a,b)$ and define
\begin{equation}
\roab \coloneqq \qbn - \frac{1}{r(a,b)}.
\end{equation}
\begin{enumerate}[label=(\roman*)]
\item For all $x,y \in (a,b)$, there is the identity
\begin{equation} 
\E_x \left[ \mc E_{H(y)} e^{- \roab H(y)}, H(y) < \sab  \right] = \hab(x) / \hab(y).
\end{equation}
\item If  $(a', b')$ is a good interval with $(a', b') \subset (a,b)$. Then,
\begin{equation}
\rho_{a',b'} < \roab < \lambda.
\end{equation}
\end{enumerate}
\end{lem}

\begin{proof}
\begin{enumerate}[label=(\roman*)]
\item It follows from Lemma \ref{LemmaMinageneralsetting}, using the argument in the proof of Proposition $3.5$ in \cite{BERT18}.
\item From $(i)$,  for all $x \in (a,b)$,
$$\E_x \left[ \mc E_{H(x)} e^{- \roab H(x)}, H(x) < \sab  \right] = 1.$$ The assertion follows than noticing that if $(a', b') \subset (a,b)$, then  $\{H(x) < \sigma(a', b') \} \subset \{H(x) < \sab \}$ and so $\p(H(x) < \sigma(a', b')) < \p(H(x) < \sab)$.
\end{enumerate}
\end{proof}

We now fix the upper-boundary point $b$ and let the lower-boundary point $a$ tend to $0$. Note that since $X$ is upward skip free, then, for all $x<b$, $\lim_{a \to 0} \sab = H(b)$ $\p_x$-almost surely. Now choose $a'$, with $0< a < a' < b$ such that
\begin{equation}
\label{conditionsceltadiaprime}
\sup_{x \in [0,a')} B(x)(N(x)-1) - \roab < - \frac{\log (1-p_{a'})}{s(0,a')},
\end{equation}
where $p_{a'}$ is the probability that the process started from $0$ reaches $a'$ without making any jump and $s(0, a')$ is defined as in Remark \ref{remarkentrancenadnaturalboundary}. As shown in Lemma \ref{Lemmahcontinuous}, this condition is clearly satisfied for $a'$ small enough, since, when $a'$ tends to $0$, $s(0, a')$ tends to $0$, while $- \log (1-p_{a'})$ tends to $+\infty$. 

\begin{prop}
\label{Propositionboundarygoestozero}
For every $a' \in (a,b)$ satisfying \eqref{conditionsceltadiaprime} and every $b'' \in (a',b)$ sufficiently close to $b$, there exists $\gamma < \lambda$ with 
\begin{equation} 
\E_{a'} \left[ \mc E_{H(a')} e^{- \gamma H(a')}, H(a') < H(b'')  \right] \in (0,1].
\end{equation}
\end{prop}

\begin{proof}
Since $(a,b)$ is a good interval, the irreducibility of the process killed when exiting $[a,b]$ implies that $p_{a'}(H(a') < H(b'')) > 0$, provided that $b''$ is chosen close enough to $b$. Then we consider the convex and non-increasing function $\Psi : \mathbb{R} \to (0, \infty]$ defined by

\begin{equation}
\Psi(q) \coloneqq \E_{a'} \left[ \mc E_{H(a')} e^{- q H(a')}, H(a') < H(b'')  \right].
\end{equation}
We already know that $\Psi (\lambda) \leq L_{a',a'}(\lambda) = 1$. Since $\roab < \lambda$, we can choose $r \in (\roab, \lambda)$. All we need to check, is that $\Psi(r) < \infty$. In fact, if $\Psi(r) \leq 1$, we choose $\gamma = r$, otherwise equation $\Psi(q) =1$ has a unique solution $\gamma \in (r, \lambda)$ by convexity.

On the event $\{H(a') < H(b'') \}$, the process remains in $[a', b'']$ until it makes a jump below $a'$ at time $\sigma(a', b'') $ and then it stays in $(0, a')$ until $H(a')$, when it hits $a'$ for the first time. 
From the Markov property, we can decompose the excursion away from $a'$ at $\sigma(a', b'')$, 
\begin{align}
 \Psi (r) & = \E_{a'} \left[ \mc E_{H(a')} e^{- r H(a')}, H(a') < H(b'')  \right] \\ & =  \E_{a'} \left[ \E \left[  \mc E_{H(a')} e^{- r H(a')}, H(a') < H(b'')   \big| \mc F_{\sigma(a', b'')} \right]\right] \\
 &= \E_{a'} \left[ \mc E_{\sigma(a', b'')} e^{- r \sigma(a', b'')} \mathbf{1}_{\{\sigma(a', b'') < H(b'')\}}  \; \E_{\sigma(a',b'')} \left[  \mc E_{H(a')} e^{- r H(a')}, H(a') < \infty \right]  \right]
\end{align}
The argument in the proof of Lemma \ref{Lemmahcontinuous} shows that, thanks to the proper choice of $a'$ made in \eqref{conditionsceltadiaprime}, 
\begin{equation}
\sup_{x \in [0, a']} \E_{x} \left[  \mc E_{H(a')} e^{- r H(a')}, H(a') < \infty \right] < \infty, 
\end{equation}
and so, there exists $C>0$ such that
\begin{equation}
 \Psi (r) \leq C \cdot \; \E_{a'} \left[ \mc E_{\sigma(a', b'')} e^{- r \sigma(a', b'')}, \; H(a')< H(b'')\right].
\end{equation}
Observe that, on the event $\{H(a') < H(b'') \}$,  there exists an instant $t \leq \sigma (a', b'')$ with  $X_t < a'$ if and only if the process $X$ stays in $[a',b'']$ during the whole time-interval $[0, t)$ and exists from $[a', b'']$ at time $t$ by jumping below $a'$. In other words, $t= \sigma(a', b'')$ and $H(a') < H(b'')$. Moreover, the predictable compensator of the jump process of $X$ is $B(X_{t-}) k(X_{t-}, y) dydt$. From this, we deduce that 
\begin{align}
\E_{a'} \left[ \mc E_{\sigma(a', b'')} e^{- r \sigma(a', b'')}, \; H(a')< H(b'')\right] & = \E_{a'} \left( \int_0^{\sigma(a', b'')} \mc E_t e^{-rt} \left( \int_0^{a'} B(X_{t-}) k(X_{t-}, y) dy \right) dt \right) \\
& \leq \norm{BN}_{\infty}  \E_{a'} \left( \int_0^{\sigma(a', b'')} \mc E_t e^{-rt} dt \right),
\end{align}
where $ \norm{BN}_{\infty}$ is the maximal jump rate. To conclude, we notice that
\begin{align}
\E_{a'} \left[ \mc E_t e^{-rt}, t < \sigma(a', b'') \right] & \leq  \frac{ e^{-(r- \roab)t} }{\min_{[a',b'']} \hab}  \E_{a'} \left[ \mc E_t e^{- \roab t} \hab(X_t), t < \sab \right] \\
& = e^{-(r- \roab)t} \frac{\hab(a')}{\min_{[a',b'']} \hab},
\end{align}
where the last equality follows from the fact that $\Mab$ defined above is a martingale and $\E_{a'} \left[ \Mab (0) \right] = \hab(a') $.
Since $\hab$ is strictly positive on $(a,b)$, the second factor is bounded and $r > \roab$ we easily get that
\begin{equation}
\int_0^{\infty} \E_{a'} \left[ \mc E_t e^{-rt}, t < \sigma(a', b'') \right] dt < \infty, 
\end{equation}
which proves the assertion.
\end{proof}

As a corollary, we get the following result. 
\begin{cor}
\label{corollariogesupermartingale}
Under the assumptions of Proposition \ref{Propositionboundarygoestozero}, for $0< x < b''$, we consider the function
\begin{equation}
g(x) \coloneqq \E_{x} \left[  \mc E_{H(a')} e^{- \gamma H(a')}, H(a') < H(b'')\right].
\end{equation}
The process
\begin{equation}
\mc S(t) \coloneqq g(X_t) \mc E_t e^{- \gamma t} \mathbf{1}_{\{ t < H(b'')  \}}, \quad t \geq 0,
\end{equation}
is then a $\p_x$-supermartingale for every $0 \leq x < b''$. 
\end{cor}

We are now ready to prove Theorem \ref{Theorem2}.
\begin{proof}[Proof of Theorem \ref{Theorem2}]
We pick two good intervals $(a,b)$ and $(a', b')$, with $0< a< a' < b' < b$ sufficiently large such that condition \eqref{conditionsceltadiaprime} is satisfied and 
\begin{equation}
\sup_{[b', \infty)} B(x) \left(N(x)-1 \right) < \roab.
\end{equation}
This is indeed possible thanks to condition \eqref{conditionlimsupatinfinity}. Next, we choose $q$ such that $\max \{ \roab, \gamma \} < q < \lambda$. In particular, 
\begin{equation}
\label{sopraabprime}
B(x) \left(N(x)-1 \right) < q \quad \text{for all} \; x \in [b', \infty).
\end{equation}
We prove that \eqref{Lbiggerthan1} holds with $x= b'$. We let $X$ start from $b'$ and we split the excursions at times $\sigma(b', \infty)$ and $\sigma(a', \infty)$. Clearly, $\p_{b'}$-almost surely, $\sigma(b', \infty) \leq \sigma(a', \infty)$. By \eqref{sopraabprime}, until time $\sigma(b', \infty)$, we have
\begin{equation}
E_{\sigma(b', \infty)} e^{- q \sigma(b', \infty)} \leq 1 \quad \p_{b'}\text{-a.s.},
\end{equation}
and so, from the strong Markov property, the assertion follows from
\begin{equation}
\label{belowb'}
\sup_{x \leq b'} \E_{x} \left[  \mc E_{H(b')} e^{- q H(b')}, H(b') < \infty \right] < \infty.
\end{equation}
First we consider the case $x \leq a'$. As in Proposition \ref{Propositionboundarygoestozero}, condition \eqref{conditionsceltadiaprime} ensures that
\begin{equation}
\sup_{x \in [0, a']} \E_{x} \left[  \mc E_{H(a')} e^{- q H(a')}, H(a') < \infty \right] < \infty, 
\end{equation}
and so, from the Markov property, 
\begin{equation}
\label{supremumsmalleraprimeisfinitepreamble}
\sup_{x \in [0, a']}\E_{x} \left[  \mc E_{H(b')} e^{- q H(b')}, H(b') < \infty \right]  \leq C \cdot \E_{a'} \left[  \mc E_{H(b')} e^{- q H(b')}, H(b') < \infty \right].
\end{equation}
To conclude the case $x \leq a'$, we thus need to show that the RHS is finite.
To this end, we choose $b'' \in (b', b)$ close enough to $b$ to have $\p_{b'} (H(a') < H(b'')) > 0$. As usual this can be done by irreducibility arguments. Recalling the notation of Corollary \ref{corollariogesupermartingale}, we have
\begin{align}
\E_{a'} \left[  \mc E_{H(b')} e^{- q H(b')}, H(b') < \infty \right] \leq \frac{1}{g(b')} \E_{a'} \left[  \mc S_{H(b')} \right] \leq \frac{g(a')}{g(b')}.
\end{align}
Proposition \ref{Propositionboundarygoestozero} ensures that $g(a') \in (0,1]$ and $g(b') > 0$ and so
\begin{align}
\label{supremumsmalleraprimeisfinite}
\sup_{x \in [0, a']}\E_{x} \left[  \mc E_{H(b')} e^{- q H(b')}, H(b') < \infty \right]  \leq \frac{g(a')}{g(b')} < \infty.
\end{align}

Finally, we consider the case $x \in (a', b')$. We distinguish whether the process exits from $[a',b']$ through the upper or the lower boundary. In the first case, Lemma \ref{expectationsinsideacompact} ensures that, for every $x \in [a',b']$,  
\begin{align}
\E_{x} \left[  \mc E_{H(b')} e^{- q H(b')}, H(b') \leq \sigma(a',b') \right] \leq \E_{x} \left[  \mc E_{H(b')} e^{- q H(b')}, H(b') < \sab \right] \leq \frac{\hab(x)}{\hab(b')},
\end{align}
and so,
\begin{align}
\label{insidetheinterval1}
\sup_{x \in [a',b')} \E_{x} \left[  \mc E_{H(b')} e^{- q H(b')}, H(b') \leq \sigma(a',b') \right] \leq \frac{\max_{[a',b']} \hab(x)}{\hab(b')} < \infty.
\end{align}
In the second case, we have that, in a similar way as in Proposition \ref{Propositionboundarygoestozero}, 
\begin{align}
\E_{x} \left[ \mc E_{\sigma(a', b')} e^{- q \sigma(a', b')}, \; \sigma(a',b')< H(b')\right] 
& \leq
\norm{BN}_{\infty}  \E_{x} \left( \int_0^{\sigma(a', b')} \mc E_t e^{-qt} dt \right) \\ 
& \leq \norm{BN}_{\infty}  \frac{\hab(x)}{(q- \roab) \min_{[a',b'']} \hab} < \infty.
\end{align}

Using the Markov property at the stopping time $\sigma(a',b')$ and using \eqref{supremumsmalleraprimeisfinite}, we conclude that 
\begin{align}
\label{insidetheinterval2}
\sup_{x \in [a',b')} \E_{x} \left[  \mc E_{H(b')} e^{- q H(b')}, \sigma(a',b') < H(b') < \infty \right] < \infty.
\end{align}
Combining \eqref{supremumsmalleraprimeisfinite}, \eqref{insidetheinterval1} and \eqref{insidetheinterval2}, we get \eqref{belowb'} and thus Theorem \ref{Theorem2} is established.
\end{proof}

\subsection{Proof of Theorem 1.2}
We now turn to the proof of Theorem \ref{Theorem3}. As stated in the Introduction, \eqref{conditionlimsupatinfinity} reduces to the more explicit criterion \eqref{conditionwhenrecurrent} when $X$ is recurrent. So, we need to find criteria in terms of the growth and fragmentation rates that ensure recurrence of $X$, when the fragmentation rate is self-similar. 
Since $X$ is irreducible, its trajectories between jumps are increasing and the jumps are only negative, point-recurrence can only fail when the paths converge almost surely to $0$ or to $\infty$. To exclude these cases, we resort to Foster-Lyapunov criteria. We refer to \cite{HAIRER16} or \cite{MT09} for a more comprehensive account. In brief, one wishes to find a smooth convex function $V: \zeroinf \to (0, \infty)$ such that
\begin{equation}
\label{definitionofV}
V(x) = \begin{cases}
x^a \quad \quad \quad x \geq x_{\infty} \\
x^{-b} \quad \quad 0< x \leq x_0,
\end{cases}
\end{equation}
for some $a,b>0$ and $0 < x_0 < x_\infty$, and such that, for all $x < x_0$ and $x > x_{\infty}$, one has
\begin{equation}
\label{thegeneratorisnegative}
\mc G V(x) \leq 0.
\end{equation}
This implies that $f\left(X_{t \wedge H(x_0)}\right)$ and $f \left(X_{t \wedge \sigma(x_{\infty}, \infty)} \right)$ are  $\p_{x}$-supermartingale respectively for all $0<x<x_0$ and all $x \geq x_\infty$, which is a sufficient condition to avoid that $X$ converges either to $0$ or to $\infty$.

When $k$ is self-similar, \textit{i.e.} it has the form \eqref{selfsimilarfragmentationkernel}, the generator is the following:
\begin{equation}
\label{definitionofgselfsimilar}
\G f (x)  =  \tau(x)f'(x) + B(x) \int_0^1 \left(f(y)- f(x) \right)  k_0(z)dz.
\end{equation}
Hence, for $0< x \leq x_0$, 
\begin{equation}
\G V(x) = x^{-b}  \left( - b \frac{\tau(x)}{x} + B(x) \int_0^1 \left(z^{-b} - 1 \right)  \rho(z)dz \right) = x^{-b}\left( - b \frac{\tau(x)}{x} + B(x) \left(M_{_{-b}} - M_{_0} \right) \right), 
\end{equation}
and, for $x \geq x_{\infty}$,
\begin{equation}
\G V(x) = x^{a} \left(  a \frac{\tau(x)}{x} + B(x) \int_0^1 \left(z^{a} - 1 \right)  \rho(z)dz \right) = x^{a} \left(  a \frac{\tau(x)}{x} + B(x) \left(M_a -M_{_0} \right) \right).
\end{equation}
This means that if \eqref{conditionrecurrencessinf} holds and
\begin{equation}
\label{conditionrecurrencesszero}
\frac{\tau(x)}{xB(x)} \geq \frac{1}{b} \left( M_{_{-b}} - M_{_0} \right)  \quad \quad \text{for all} \; x \leq x_0,
\end{equation}
then $X$ is recurrent. Condition \eqref{conditionrecurrencesszero} is directly verified under our assumptions, since \eqref{Bcontinuousbounded} and \eqref{Tfinite} implies that $\lim_{x \to 0}  \tau(x)/xB(x) = \infty$ and thus \eqref{assmoments} and \eqref{conditionrecurrencessinf} are enough to ensure point recurrence of $X$. We already argued that condition \eqref{conditionwhenrecurrent} ensures \eqref{conditionlimsupatinfinity} when $X$ is recurrent, and so the claim follows applying Theorem \ref{Theorem2}.

\section{More examples}
\label{Section5}

Here we deal with the case in which $B(x)=B >0$ and $N(x)=N>0$ are constant. The generator of the process $X$ is then
\begin{equation}
\label{definitionofgBconstant}
\G f (x) = \tau(x)f'(x) + B \int_0^x \left(f(y)- f(x) \right)  k(x,y) dy,
\end{equation}
with $\int_0^x k(x, y) dy =N$ for all $x\geq0$. The Feynman-Kac formula gives
\begin{equation}
T_t f (x) = e^{B(N-1)t} \; \E_x \left[  f(X_t) \right].
\end{equation}
If $X$ is recurrent, then, $\p_x \left( H(x) < \infty \right) = 1$ and \eqref{Lequal1} holds with $\lambda = B(N-1)$. In this case, we cannot rely on criterion \eqref{conditionwhenrecurrent} to prove exponential convergence, as $\lim_{x \to \infty} B(x)(N(x)-1)= B(N-1) = \lambda$.

However, if we can find find sufficient conditions for $X$ to be positive recurrent, then it has a (unique) stationary distribution $\nu$ and the convergence 
\begin{equation}
\lim_{t \to \infty} e^{- B(N-1)t} T_t f(x) = \langle \nu, f \rangle, 
\end{equation}
holds for all continuous functions with compact support. If $X$ is further exponentially ergodic, then \eqref{definitionexponentialconvergence} holds. We resort again to Foster-Lyapunov techniques. We define, for $s \in \mathbb{R}$,
\begin{equation}
M_x(s) \coloneqq \frac{1}{B} \int_0^x (y/x)^{s} k(x,y) dy  \quad \quad \text{and} \quad \quad M(s) \coloneqq \sup_{x>0} M_x(s),
\end{equation}
and we assume that
\begin{equation}
\label{assumptionmomentsconstant} 
\text{there exists} \; a >0 \; \text{and} \; b>0 \; \text{such that} \; M(a) < M(0) \; \text{and} \; M_{-b} < \infty.
\end{equation}
Note that $M$ is decreasing and that $N = M(0)$. If $V$ is as in \eqref{definitionofV}, we have 
\begin{equation}
\G V(x) = \begin{cases} x^{-b} \left( -b \frac{\tau(x)}{x} + B  \left( M(-b) - M( \right) \right) \quad \quad x \leq x_0 \\
x^{a} \left( a \frac{\tau(x)}{x} + B  \left(- M_a - M_{_0} \right) \right) \quad \quad \quad x \geq x_\infty.
\end{cases}
\end{equation}

We already argued that if $\mc G V(x) \leq 0$, then $X$ is point recurrent. If one can further find $\alpha>0$, $0< \alpha' < \infty$ and $K$ compact in $\zeroinf$ such that
\begin{equation}
\label{conditionexpharrisrecurrence}
\G V(x) \leq - \alpha V(x) + \alpha' \mathbf{1}_{K},
\end{equation}
then $X$ is exponentially ergodic. This happens if the two conditions
\begin{equation}
\label{conditionexpharrisrecurrencezero}
\frac{\tau(x)}{x} \geq \frac{B}{b}  \left( M_{_{-b}} - M_{_0} \right) \quad \quad x \leq x_0,
\end{equation}
and
\begin{equation}
\label{conditionexpharrisrecurrenceinf}
 \frac{\tau(x)}{x} \leq \frac{B}{a} \left( M_{_0} - M_a \right) \quad \quad x \geq x_\infty,
\end{equation}
hold. Notice that 
\eqref{conditionexpharrisrecurrencezero} is directly verified as soon as the moment $M_{-b}$ is defined, since \eqref{Tfinite} implies that $\lim_{x \to 0}  \tau(x)/x = \infty$. On the other hand, \eqref{conditionexpharrisrecurrenceinf} seems natural, as a bound on the growth is expected when the fragmentations are bounded. Moreover, the bound is not too restrictive, as we are already assuming \eqref{Timetoinfinity}. To sum up, we have the following result. 
\begin{prop}
Assume \eqref{tauLipschitz}, \eqref{kcontinuous},  \eqref{Timetoinfinity}, \eqref{conditionequivalenttoirreducibility} and \eqref{vagueconvergenceofthekernel}. Assume that $B(x)=B$ and $N(x)=N$ for all $x>0$ and $B, N >0$. If \eqref{assumptionmomentsconstant} and \eqref{conditionexpharrisrecurrenceinf} hold, then there is Malthusian behaviour with exponential speed of convergence.
\end{prop}

\begin{oss}
We know that if $\mc A h = \lambda h$ for some $\lambda$, than $\mc H f(x) = h(x)^{-1} \mc A (hf)(x) - \lambda f$
is the generator of a Markov process, say $Z$.
Then,
\begin{equation}
T_t f(x) = e^{\lambda t} h(x)  \E_x \left[  f(Z_t)/ h(Z_t) \right].
\end{equation}
When $B$ and $N$ are constant, $\mc A 1= B(N-1) 1$, and the above formula holds with $\lambda = B(N-1)$ and $h=1$.  We conclude noticing that the same value for the Malthus exponent was obtained by Bertoin and Watson the case in which the kernel is homogeneus and there is conservation of mass (see Chapter $7$ of \cite{BW18}).
\end{oss}

	\bibliographystyle{alpha}	
	\bibliography{biblio.bib}

\end{document}